\documentclass[12pt]{amsart}
\usepackage{cases}
\usepackage{mathrsfs}
\usepackage{color}
\usepackage{latexsym,amssymb}
\usepackage{a4wide}
\usepackage{amsthm}
\usepackage{amsmath}
\usepackage{graphicx}
\input xy
  \xyoption{all}
\usepackage{verbatim}
\usepackage[lite,abbrev,msc-links]{amsrefs}
\usepackage{framed}
\usepackage{bm}

\numberwithin{equation}{section}
\begin{document}
\theoremstyle{plain}
\newtheorem{thm}{Theorem}[section]
\newtheorem{lem}[thm]{Lemma}
\newtheorem{cor}[thm]{Corollary}
\newtheorem{cor*}[thm]{Corollary*}
\newtheorem{prop}[thm]{Proposition}
\newtheorem{prop*}[thm]{Proposition*}
\newtheorem{conj}[thm]{Conjecture}
\theoremstyle{definition}
\newtheorem{construction}{Construction}
\newtheorem{notations}[thm]{Notations}
\newtheorem{question}[thm]{Question}
\newtheorem{prob}[thm]{Problem}
\newtheorem{rmk}[thm]{Remark}
\newtheorem{remarks}[thm]{Remarks}
\newtheorem{defn}[thm]{Definition}
\newtheorem{claim}[thm]{Claim}
\newtheorem{assumption}[thm]{Assumption}
\newtheorem{assumptions}[thm]{Assumptions}
\newtheorem{properties}[thm]{Properties}
\newtheorem{exmp}[thm]{Example}
\newtheorem{comments}[thm]{Comments}
\newtheorem{blank}[thm]{}
\newtheorem{observation}[thm]{Observation}
\newtheorem{defn-thm}[thm]{Definition-Theorem}
\newtheorem{defn-lem}[thm]{Definition-Lemma}
\newtheorem*{Setting}{Setting}

\newcommand{\sA}{\mathscr{A}}
\newcommand{\sB}{\mathscr{B}}
\newcommand{\sC}{\mathscr{C}}
\newcommand{\sD}{\mathscr{D}}
\newcommand{\sE}{\mathscr{E}}
\newcommand{\sF}{\mathscr{F}}
\newcommand{\sG}{\mathscr{G}}
\newcommand{\sH}{\mathscr{H}}
\newcommand{\sI}{\mathscr{I}}
\newcommand{\sJ}{\mathscr{J}}
\newcommand{\sK}{\mathscr{K}}
\newcommand{\sL}{\mathscr{L}}
\newcommand{\sM}{\mathscr{M}}
\newcommand{\sN}{\mathscr{N}}
\newcommand{\sO}{\mathscr{O}}
\newcommand{\sP}{\mathscr{P}}
\newcommand{\sQ}{\mathscr{Q}}
\newcommand{\sR}{\mathscr{R}}
\newcommand{\sS}{\mathscr{S}}
\newcommand{\sT}{\mathscr{T}}
\newcommand{\sU}{\mathscr{U}}
\newcommand{\sV}{\mathscr{V}}
\newcommand{\sW}{\mathscr{W}}
\newcommand{\sX}{\mathscr{X}}
\newcommand{\sY}{\mathscr{Y}}
\newcommand{\sZ}{\mathscr{Z}}
\newcommand{\bZ}{\mathbb{Z}}
\newcommand{\bN}{\mathbb{N}}
\newcommand{\bO}{\mathbb{O}}
\newcommand{\bQ}{\mathbb{Q}}
\newcommand{\bC}{\mathbb{C}}
\newcommand{\bR}{\mathbb{R}}
\newcommand{\bH}{\mathbb{H}}
\newcommand{\bD}{\mathbb{D}}
\newcommand{\bE}{\mathbb{E}}
\newcommand{\bV}{\mathbb{V}}
\newcommand{\bfp}{\mathbf{p}}
\newcommand{\bft}{\mathbf{t}}
\newcommand{\bfM}{\mathbf{M}}
\newcommand{\bfN}{\mathbf{N}}
\newcommand{\bfX}{\mathbf{X}}
\newcommand{\bfY}{\mathbf{Y}}
\newcommand{\cH}{\mathcal{H}}
\newcommand{\cV}{\mathcal{V}}
\newcommand{\spec}{\textrm{Spec}}
\newcommand{\dbar}{\bar{\partial}}
\newcommand{\ddbar}{\partial\bar{\partial}}
\newcommand{\redref}{{\color{red}ref}}

\title[$L^2$ representation of Simpson-Mochizuki's prolongation] {$L^2$ representation of Simpson-Mochizuki's prolongation of Higgs bundles and the Kawamata-Viehweg vanishing theorem for semistable parabolic Higgs bundles}


\author[Chen Zhao]{Chen Zhao}
\email{czhao@ustc.edu.cn}
\address{School of Mathematical Sciences,
	University of Science and Technology of China, Hefei, 230026, China}

\begin{abstract}
In this paper, we provide an $L^2$ fine resolution of the prolongation of a nilpotent harmonic bundle in the sense of Simpson-Mochizuki (an analytic analogue of the Kashiwara-Malgrange filtrations). This is the logarithmic analogue of Cattani-Kaplan-Schmid's and Kashiwara-Kawai's results on the $L^2$ interpretation of the intersection complex. As an application, we give an $L^2$-theoretic proof to the Nadel-Kawamata-Viehweg vanishing theorem with coefficients in a nilpotent Higgs bundle.
\end{abstract}

\maketitle
\tableofcontents

\section{Introduction}
\subsection{Main result}
It has been long realized that the (nonabelian) Hodge theory is closely related to the $L^2$-differential forms. Besides the proof of the classical Hodge decomposition theorem using $L^2$ methods, there are many interesting and deep relations between the various extensions of the polarized variation of Hodge structure and the $L^2$-de Rham complexes. For the $L^2$ interpretation of the intersection complex, readers may see  Cattani-Kaplan-Schmid \cite{Cattani_Kaplan_Schmid1987}, Kashiwara-Kawai \cite{Kashiwara_Kawai1987}, E. Looijenga \cite{Looijenga1988}, Saper-Stern \cite{Saper_Stern1990}, S. Zucker \cite{Zucker1979} and Shentu-Zhao \cite{SC2021_CGM}. The main purpose of the present paper is to investigate the relavant problem in the context of the nonabelian Hodge theory. The main result is an $L^2$-fine resolution of the Simpson-Mochizuki's prolongation (an analytic analogue of the Kashiwara-Malgrange filtrations \cite{Kashiwara1983}) of nilpotent Higgs bundles (e.g. polarized complex variation of Hodge structure). Before stating the main result, let us fix some notations.

Let $X$ be a pre-compact open subset of a hermitian manifold $(M,\omega_M)$ and let $D=\cup_{i=1}^l D_i$ be a simple normal crossing divisor on $M$. Let $\sigma_i\in H^0(M,\sO_M(D_i))$ be the defining section of $D_i$ and $h_{i}$ an arbitrary hermitian metric on $\sO_M(D_i)$. We denote $\varphi_i=|\sigma_i|_{h_{i}}$. Let $\phi_P\in C^\infty(X\backslash D)$ such that $\omega_P:=\sqrt{-1}\ddbar\phi_P+\omega_M|_{X\backslash D}$ is a hermitian metric on $X\backslash D$ which has Poincar\'e type growth near $D\cap X$. Let $(E,\theta,h)$ be a nilpotent harmonic bundle on $M\backslash D$ together with $\dbar_E$ its holomorphic structure. For every indices ${\bf a}=(a_1,\dots,a_l)\in\bR^l$, denote $_{\bf a}E$ to be the associated prolongation of $(E,\theta,h)$ in the sense of Simpson \cite{Simpson1990} and Mochizuki \cite{Mochizuki20072,Mochizuki20071}. It is a locally free $\sO_X$-module consisting of the holomorphic local sections $s$ of $E$ such that $|s|_h\lesssim\prod_{i=1}^l\varphi_i^{-a_i-\epsilon}$ for every $\epsilon>0$. The restricted Higgs field $\theta|_{_{\bf a}E}$ has at most logarithmic poles along $D$. We denote by $${\rm Dol}(_{\bf a}E,\theta):= {_{\bf a}}E\to {_{\bf a}}E\otimes\Omega_{X}(\log D)\to\cdots$$
the associated logarithmic Dolbeault complex. The main purpose of the present  paper is to construct a complex of fine sheaves (consisting of certain locally square integrable  differential forms) that is canonically quasi-isomorphic to ${\rm Dol}(_{\bf a}E,\theta)$.

The set of prolongations $\{_{\bf a}E\}_{{\bf a}\in\bR^l}$ forms a parabolic bundle. In particular, the set of jumping indices along each component $D_i$
$$\sJ_{D_i}(E):=\left\{a\in\bR\big|{\rm Gr}^{(i)}_{{\bf a}\uparrow_ia}E:=_{{\bf a}\uparrow_ia}E/\cup_{b<a}{_{{\bf a}\uparrow_ib}}E\neq0,\forall {\bf a}\in\bR^l\right\}$$
is a discrete subset of $\bR$ and $\sJ_{D_i}(E)+\bZ=\sJ_{D_i}(E)$. Here we denote ${\bf a}\uparrow_i c$ by the indices obtained by deleting the $i$th component of $\bm{a}$ and replacing it with the number $c$.

For $N\in \bZ$, denote $h_N({\bf a}):=he^{-N\phi_P+\sum_{i=1}^l2a_i\log\varphi_i}$ to be the modified metric on $E$. Define the operator $D''=\dbar_E+\theta$. Denote by $\sD^{k}_{X,\omega_P}(E,D'',h_N({\bf a}))$ the sheaf of measurable $E$-valued $k$-forms $\alpha$ such that $\alpha$ and $D''\alpha$ are locally square integrable near every point of $X$ with respect to $\omega_P$ and $h_N(\bm{a})$.  Denote  
$$\sD^{\bullet}_{X,\omega_P}(E,D'',h_N({\bf a})):=\sD^{0}_{X,\omega_P}(E,D'',h_N({\bf a}))\stackrel{D''}{\to}\sD^{1}_{X,\omega_P}(E,D'',h_N({\bf a}))\stackrel{D''}{\to}\cdots$$
to be the associated $L^2$-Dolbeault complex. There is a natural inclusion $${\rm Dol}(_{\bf a}E,\theta)\to \sD^{\bullet}_{X,\omega_P}(E,D'',h_N({\bf a}+\bm{\epsilon}))$$ for every $(0,\dots,0)=:\bm{0}<\bm{\epsilon}\in\bR^l$\footnote{By $(a_1,\dots,a_l)<(b_1,\dots,b_l)$ we mean that $a_i<b_i$ for every $i=1,\dots,l$.}.

For every ${\bf a}=(a_1,\dots,a_l)\in\bR^l$, define
$$\sigma_{E,i}(a_i)=\min\{|b-a_i||b\in \sJ_{D_i}(E), b\neq a_i\},\quad\forall 1\leq i\leq l$$
and $\bm{\sigma}_E({\bf a})=(\sigma_{E,1}(a_1),\dots,\sigma_{E,l}(a_l))$.
The main result of the present paper is
\begin{thm}\label{thm_main}
	There is a constant $N_0$,  depending only on $(E,\theta,h)$ and $X$ (independent of $\bm{a}$ and $\bm{\epsilon}$ in the following), such that  the inclusion map $${\rm Dol}( _{{\bf a}}E,\theta)\to \sD^{\bullet}_{X,\omega_P}(E,D'',h_N({\bf a}+\bm{\epsilon}))$$ is a quasi-isomorphism for every
	${\bf a}\in\bR^l$, every $N>N_0$ and every $\bm{0}<\bm{\epsilon}<\bm{\sigma}_E({\bf a})$. 
\end{thm}



\subsection{Application: vanishing theorems}
Besides its own insterests, Theorem \ref{thm_main} allows us to prove the following Kawamata-Viehweg type vanishing theorem for stable parabolic higgs bundles. 
\begin{thm}[=Theorem \ref{thm_main_relative}]\label{thm_main1}
Let $X$ be a  projective manifold of dimension $n$, $Y$ be an analytic space and $f:X\rightarrow Y$ be a proper surjective holomorphic map. Let $D=\sum_{i=1}^l D_i$ be a normal crossing divisor on $X$ and $(\{{_\alpha}E\},\theta,h)$ a stable parabolic higgs bundles on $X\backslash D$ where $\theta$ is nilpotent. Let $L$ be a holomorphic line bundle on $X$.
Suppose that some positive multiple $mL=A+F_1+F_2$ where $A$ is an ample line bundle, $F_1$ and $F_2$ are effective divisors supported in $D$ such that $A+F_1$ is nef. Suppose that $F_2=\sum_{i=1}^l r_i D_i$ with $r_i\in \bZ_{\geq 0},\forall i$. Denote ${\bf{r}}:=(r_1,\dots,r_l)$.

Then
\begin{align}\label{align_main_KVvanishing}
	R^if_\ast({\rm Dol}(_{-\frac{\bf{r}}{m}+{\bf a}}E,\theta)\otimes L)=0
\end{align}
for any $i>n$ and any ${\bf a}=(a_1,\dots,a_l)\in \bR^l$ such that $|a_j|<\sigma_{E,j}(-\frac{r_j}{m})$ for every $j=1,\dots,l$.
\end{thm}
\begin{rmk}
	The nilpotentness condition in Theorem \ref{thm_main1} is not necessary. By using the trick of \cite{AHL2019} we could deduce the relavant vanishing results for an arbitrary tame harmonic bundle to Theorem \ref{thm_main1}.
\end{rmk}
Theorem \ref{thm_main1} generalizes several recent vanishing results. When $Y$ is a point and $L$ is ample, Theorem \ref{thm_main1} is proved by Arapura-Hao-Li \cite{AHL2019} and Deng-Hao \cite{DF2021}, which extends Arapura's logarithmic Saito-Kodaira vanishing theorem for complex polarized variations of Hodge structure with unipotency condition in \cite{Arapura2019}. When $Y$ is a point, ${\bm a}={\bm 0}$ and $F_2=0$, Theorem \ref{thm_main1} implies J. Suh's vanishing theorem for the canonical extension of a polarizable variation of Hodge structure \cite{Suh2018} which provides many interesting vanishing results on Shimura varieties. We will come back to this issue in \S \ref{section_Saito_vanishing}.

The most interesting phenomenon in Theorem \ref{thm_main1}, compared to other vanishing theorems, is the appearance of the parameter ${\bf a}\in \bR^l$. There may exist more than one vanishing result in Theorem \ref{thm_main1} depending on whether $-\frac{r_i}{m}\in\sJ_{D_i}(E)$ for some $i\in\{1,\dots,l\}$.
Assume that $$\sJ_{D_i}(E)=\{\cdots <c_{i,j-1}<c_{i,j}<c_{i,j+1}<\cdots\},\quad\forall 1\leq i\leq l.$$ 
For every  $1\leq i\leq l$, we denote 
$$\sJ_{E,D_i}(b)=\begin{cases}
	\{c_{i,j-1},c_{i,j}\}, & b=c_{i,j} \\
	\{c_{i,j}\}, & c_{i,j}<b<c_{i,j+1}.
\end{cases}$$
Set $$\sJ_{E,D}({\bf a})=\prod_{i=1}^l \sJ_{E,D_i}(a_i)$$ for each ${\bf a}=(a_1,\dots,a_l)\in \bR^l$. The vanishing result (\ref{align_main_KVvanishing}) is equivalent to the family of vanishing results
\begin{align}
	R^if_\ast({\rm Dol}(_{{\bf a}}E,\theta)\otimes L)=0,\quad \forall i>n,\quad {\bf a}\in \sJ_{E,D}(-\frac{\bf{r}}{m}).
\end{align}
\subsubsection{Kawamata-Viehweg vanishing theorem and its $(i,j)$ version}
Notations as in Theorem \ref{thm_main1}.
For an easy example, let us consider the trivial Higgs bundle $\sO_{X\backslash D}$ with vanishing Higgs field. In this case $\sJ_{D_i}(\sO_{X\setminus D})=\bZ$ for every $i=1,\dots,l$. Hence 
\begin{align}\label{align_KV_indices}
\sJ_{E,D_i}(-\frac{r_i}{m})=\begin{cases}
\{-\frac{r_i}{m}-1,-\frac{r_i}{m}\}, & \frac{r_i}{m}\in\bZ \\
\{\lfloor-\frac{r_i}{m}\rfloor\} & \frac{r_i}{m}\notin\bZ
\end{cases}.
\end{align}
Let $B$ be a big and nef $\bQ$-divisor such that $L=\lceil B\rceil$ and $F_2=\lceil B\rceil-B$. In this case $0\leq\frac{r_i}{m}<1$ for every $i=1,\dots,l$.
Since
$${\rm Dol}({_{\bf b}}E,\theta)\simeq\bigoplus_{j=0}^n\Omega^{j}_{X}(\log D)(\sum_{i=1}^l \lfloor b_i\rfloor D_i)[-j],\quad\forall\bm{b}=(b_1,\dots,b_l)\in\bR^l,$$
we obtain the following Kawamata-Viehweg type vanishing theorem.
\begin{cor}\label{cor_KV_pq}
	Let $X$ be a  projective manifold of dimension $n$, $Y$ be an analytic space and $f:X\rightarrow Y$ be a proper surjective holomorphic map. Let $D=\sum_{i=1}^l D_i$ be a normal crossing divisor on $X$. Let 
	$B=A+F$ be a $\bQ$-divisor where $B$ is nef, $A$ is an ample $\bQ$-divisor, ${\rm supp}(F)\subset D$ and ${\rm supp}(\lceil B\rceil-B)\subset D$. 
	Then
	\begin{align}\label{align_rel_KV}
	R^if_\ast(\Omega^{j}_{X}(\log D)\otimes{\sO_X}(\lceil B\rceil-G))=0,\quad \forall i+j>n
	\end{align}
	for every reduced effective divisor $G$ such that ${\rm supp}(\lceil B\rceil-B)\subset G\subset D$. 
\end{cor}
When $j=n$, by standard arguments Corollary \ref{cor_KV_pq} implies the classical Kawamata-Viehweg vanishing theorem for $\bQ$-divisors \cite[Theorem 2.64]{Mori1998} which is widely used in birational geometry. 
When $B$ is ample and $Y$ is algebraic, (\ref{align_rel_KV}) has been recently  proved by Arapura-Matsuki-Patel-W{\l}odarczyk \cite{Arapura2018}. 
\subsubsection{Saito-type vanishing theorems}\label{section_Saito_vanishing}
Notations as in Theorem \ref{thm_main1}.
Assume that $(E,\theta,h)$ is associated with a polarized $\bC$-variation of Hodge structure $(\cV,\nabla,\{F^\bullet\},h)$ via Simpson's correspondence \cite{Simpson1988}. For every ${\bm{a}}=(a_1,\dots,a_l)\in\bR^l$, let $\cV_{\geq\bf{a}}$ be the unique locally free $\sO_X$-module extending $\cV$ such that $\nabla$ induces a connection with logarithmic singularities 
$\nabla:\cV_{\geq{\bm{a}}}\to\cV_{\geq{\bm{a}}}\otimes\Omega_X(\log D)$ whose real parts of the eigenvalues of the residue of $\nabla$ along $D_i$ belong to $[a_i,a_i+1)$. Let $j:X\backslash D\to X$ be the open immersion.
Denote $F^p_{\geq{\bm{a}}}:=j_\ast F^p\cap\cV_{\geq{\bm{a}}}$.
The recent work of Deng \cite{Deng2022} shows that $F^{p}_{\geq{\bm{a}}}$ is a subbundle of $\cV_{\geq{\bm{a}}}$ and $\nabla$ induces a complex
$$F^{p}_{\geq{\bm{a}}}\stackrel{\nabla}{\to} F^{p-1}_{\geq{\bm{a}}}\otimes\Omega_{X}(\log D)\stackrel{\nabla}{\to}F^{p-2}_{\geq{\bm{a}}}\otimes\Omega^2_{X}(\log D)\stackrel{\nabla}{\to}\cdots,\quad\forall p.$$
Denote the graded quotient complex as 
$$_{\bf a}{\rm DR}_{(X,D)}(\cV,F^\bullet):=\bigoplus_p F^{p}_{\geq{\bm{a}}}/F^{p+1}_{\geq{\bm{a}}}\stackrel{{\rm Gr}\nabla}{\to}\bigoplus_p F^{p-1}_{\geq{\bm{a}}}/F^{p}_{\geq{\bm{a}}}\otimes\Omega_{X}(\log D)\stackrel{{\rm Gr}\nabla}{\to}\cdots.$$
By \cite{Simpson1988} one has $E=Gr_{\sF}(\cV)$ and $\theta=Gr_{\sF}(\nabla)$.
We obtain that 
\begin{align}
	_{\bf a}{\rm DR}_{(X,D)}(\cV,F^\bullet)\simeq{\rm Dol}(_{\bm{a}}E,\theta),\quad\forall \bm{a}\in\bR^l.
\end{align}
Thus Theorem \ref{thm_main1} implies the following Saito-type vanishing result. 
	\begin{cor}\label{cor_main_suh}
	Let $X$ be a  projective manifold of dimension $n$, $Y$ be an analytic space and $f:X\rightarrow Y$ be a proper surjective holomorphic map. Let $D=\sum_{i=1}^l D_i$ be a normal crossing divisor on $X$. Let $(\cV,\nabla,\{F^\bullet\},h)$ be a polarizable $\bC$-variation of Hodge structure on $X\backslash D$. Let $L$ be a holomorphic line bundle on $X$.
	Suppose that some positive multiple $mL=A+F_1+F_2$ where $A$ is an ample line bundle, $F_1$ and $F_2$ are effective divisors supported in $D$ such that $A+F_1$ is a nef holomorphic line bundle. Suppose that $F_2=\sum_{i=1}^l r_i D_i$ with $r_i\in \bZ_{\geq 0},\forall i$. Denote ${\bf{r}}:=(r_1,\dots,r_l)$.
	
	Then
	$$R^if_\ast\left(_{-\frac{\bf{r}}{m}+{\bf a}}{\rm DR}_{(X,D)}(\cV,F^\bullet)\otimes L\right)=0$$
for any $i>n$ and any ${\bf a}=(a_1,\dots,a_l)\in \bR^l$ such that $|a_j|<\sigma_{E,j}(-\frac{r_j}{m})$ for every $j=1,\dots,l$.
\end{cor}
When $Y$ is a point, $F_2\simeq\sO_X$, ${\bf a}={\bf 0}$ and $(\cV,\nabla,\{F^\bullet\},h)$ is an $\bR$-polarized variation of Hodge structure, it is proved by J. Suh \cite{Suh2018} by using the theory of Hodge modules. As noted above, there are actually many other vanishing theorems in the setting of Corollary \ref{cor_main_suh} when $-\frac{r_i}{m}$ is the jumping index with respect to $D_i$ for some $i$. The jumping indices of the set of prolongations $\{{_{\bm{a}}}E|{\bm{a}}\in\bR^l\}$ are closely related to the monodromies of the flat bundle $(\cV,\nabla)$. Let ${\rm Res}_i(\nabla)$ be a residue of the monodromy of $(\cV,\nabla)$ along $D_i$ whose real parts of the eigenvalues lie in $[0,1)$. Let 
$$J_i=\{\textrm{the real part of an eigenvalue of the residue of }(\cV,\nabla) \textrm{ along }D_i\}$$
By Simpson's table \cite[p. 720]{Simpson1990}, one obtains that 
\begin{align}
	\sJ_{D_i}(E)=J_i+\bZ.
\end{align}

	{\bf Notations:}
	\begin{itemize}
\item For a holomorphic vector bundle $E$ on a complex manifold $X$, we denote $\dbar_E$ to be its holomorphic structure. We write it as $\dbar$ if no ambiguity appears.
\item 	We say that two metrics $ds_1^2$ and $ds_2^2$ are quasi-isometric, i.e., $ds_1^2\sim ds_2^2$, if there exists a constant $C>0$ such that $C^{-1}ds_1^2\leq ds_2^2\leq Cds_1^2$.  Two hermitian metrics $h$ and $h'$ of a holomorphic vector bundle on $X$ are called quasi-isometric if there exits a contant $C>0$ such that $C^{-1}h<h'<Ch$. We denote it by $h\sim h'$.
	\item Let $\alpha$ and $\beta$ be functions, metrics or $(1,1)$-forms. We denote $\alpha\lesssim\beta$ if $\alpha\leq C\beta$ for some $C\in\bR_{>0}$. We say that $\alpha$ and $\beta$ are quasi-isometric if $\alpha\lesssim\beta$ and $\beta\lesssim\alpha$. We denote it by $\alpha\sim \beta$.
    \item For a hermitian vector bundle $(E,h)$ on a complex manifold $X$, we always denote $R(E,h)$ (or simply $R(h)$) to be its Chern curvature.
    \item 
    	Let $\Theta\in A^{1,1}(X, {\rm End}(E))$ be a real form. Assume locally that
    	$$\Theta=\sqrt{-1}\sum_{i,j}\omega_{ij}e_i\otimes e_j^\ast$$
    	where $\omega_{ij}\in A^{1,1}_{X}$, $(e_1,\dots,e_r)\in E$ is an orthogonal local frame of $E$ and $(e^\ast_1,\dots,e^\ast_r)\in E^\ast$ is the dual frame.
    	We say that $\Theta$ is Nakano semi-positive, denoting $\Theta\geq_{{\rm Nak}}0$, if the bilinear form
    	\begin{align*}
    	\theta(u_1,u_2):=\sum_{i,j}\omega_{ij}(u_{1i},\overline{u_{2j}}),\quad{\textrm{where}}\quad u_l=\sum_{i}u_{lk}\otimes e_k\in T_{X}\otimes E,\quad \forall l=1,2
    	\end{align*}
    	is semi-positive definite.
    	
    	Let $\Theta_1,\Theta_2\in A^{1,1}(X, {\rm End}(E))$ be two real forms. We denote $\Theta_1\geq_{{\rm Nak}}\Theta_2$ if $\Theta_1-\Theta_2\geq_{{\rm Nak}}0$. $(E,h)$ is Nakano semi-positive if $\sqrt{-1}R(E,h)\geq_{{\rm Nak}}0$.
	\end{itemize}

\section{Preliminary}
In this section, we review the basic knowledge of Higgs bundles and nilpotent harmonic bundles.  References include \cite{Simpson1988,Simpson1990,Simpson1992,Mochizuki20071,Mochizuki20072}  and so on.
\subsection{Nilpotent harmonic bundles}\begin{defn}
 Let $X$ be a complex manifold. A Higgs bundle on $X$
	is a pair $(E,\theta)$ where $E$ is a holomorphic vector bundle on $X$ together with $\dbar_E$ its holomorphic structure, and  $\theta:E\rightarrow E\otimes \Omega_X^1$ is a holomorphic one form
	such that $\theta\wedge \theta=0$ in $\rm End$$(E)\otimes \Omega_X^2$.
	Here $\theta$ is called the Higgs field.
\end{defn}

Let $X$ be a complex manifold and let $(E,\theta)$ be a Higgs bundle on $X$. Define an operator $D''=\dbar_E+\theta$. Then $D''^2=0$. Consider a smooth hermitian metric $h$ on $E$. Let $\partial_h+\dbar_E$ be the Chern connection associated to $h$ and denote $\theta_h^\ast$ to be the adjoint of $\theta$ with respect to $h$. Denote $D_h':=\partial_h+\theta_h^\ast$.
\begin{defn}
	A smooth hermitian metric $h$ on a Higgs bundle $(E,\theta)$ is called harmonic if the operator $D_h:=D_h'+D''$ is integrable, that is, $D_h^2=0$. A harmonic bundle is a Higgs bundle endowed with a harmonic metric.
\end{defn}
Let $(E,\theta,h)$ be a harmonic bundle. Then one has the self-dual equation 
\begin{align}\label{align_selfdual_equ}
R(h)+[\theta,\theta_h^\ast]=0.
\end{align}

Let $X$ be an $n$-dimensional complex manifold and $D=\cup_{i\in I}D_i$ a simple normal crossing divisor on $X$. 

\begin{defn}[Admissible coordinate]
Let $p$ be a point of $X$ and $\{D_j\}_{j=1,\dots,l}$  the components of $D$ containing $p$. An
admissible coordinate around $p$ is the tuple $(U, \varphi)$ (or $(U;z_1,\dots,z_n)$ if no ambiguity appears) such that
\begin{enumerate}
	\item $U$ is an open subset of $X$ containing $p$.
	\item $\varphi$ is a biholomorphic morphism $U\rightarrow \Delta^n=\{(z_1,\dots,z_n)\in\bC^n\big||z_i|<1,\forall 1\leq i\leq n\}$ such that $\varphi(p) = (0,\dots,0)$ and $\varphi(D_j)= \{z_j = 0\}$ for any $j=1,\dots,l$.
\end{enumerate}
\end{defn}

\begin{defn}\label{defn_tamenilpotent}
	Let $(E,\theta,h)$ be a harmonic bundle of rank $r$ defined  on $X\backslash D$.
	Let $p$ be any point of $X$ and $(U, \varphi )$ an admissible coordinate around $p$. On $U$, we have
	the description:
	\begin{align}\label{align_theta}
	\theta=\sum_{j=1}^l f_j\cdot d\log  z_j+\sum_{k=l+1}^n g_k\cdot dz_k.
	\end{align}
	\begin{enumerate}
		\item (tameness) Let $t$ be a formal variable. We have the polynomials $\textrm{det}(t-f_j )$ and $\textrm{det}(t-g_k )$ of $t$, whose
		coefficients are holomorphic functions defined over $U\backslash{\cup_{j=1}^lD_j}$. When the functions are extended to
		the holomorphic functions over $U$, the harmonic bundle is called tame at $p$.
		A harmonic bundle is called tame if it is tame at every point $p\in X$.
		\item (nilpotentness) When $\textrm{det}(t-f_j)|_{U\cap D_{j}}=t^r$, the
		harmonic bundle is called nilpotent at $p$.
		When $(E, h, \theta)$ is nilpotent at any point $p\in X$,  it is called a nilpotent harmonic
		bundle.
		
	\end{enumerate}
\end{defn}

\subsection{Boundedness for the Higgs field}
Denote $\Delta$ (resp. $\Delta^\ast$) to be the unit disc (resp. punctured unit disc) in $\bC$. A Poincar\'e metric $\omega_P$ on $(\Delta^\ast)^l\times \Delta^{n-l}$ is defined as
$$\omega_P=\sum_{j=1}^l \frac{\sqrt{-1}dz_j\wedge d\bar{z}_j}{|z_j|^2({\log}|z_j|^2)^2}+\sum_{k=l+1}^n \sqrt{-1}dz_k\wedge d\bar{z}_k.$$

It can be defined by a potential function as
$$\omega_P=-\sqrt{-1}\partial\dbar \log(\prod_{j=1}^l(-\log|{z}_j|^2))+\sqrt{-1}\partial\dbar\sum_{k=l+1}^n |z_k|^2.$$

\begin{defn}[Poincar\'e type metric]
	 A  metric $ds^2$ on $X\backslash D$ is said to have Poincar\'e type growth near the divisor $D$ if, for every point $p\in D$ there is a coordinate neighborhood $U_p\subset X$ of $p$ with $U_p\cap (X\backslash D)\simeq(\Delta^\ast)^l\times \Delta^{n-l}$ for some $1\leq l\leq n$ such that in these coordinates, $ds^2$ is quasi-isometric to $\omega_P$.
	 
	 We will always denote $\omega_P$ to be the Poincar\'e type metric if no ambiguity causes.
\end{defn}

\begin{prop}[\cite{CG1975,Zucker1979}]\label{prop_poincare}
The Poincar\'e type metric	has finite volume, bounded curvature tensor and bounded covariant derivatives.
	
	\end{prop}

For nilpotent harmonic bundles, the following important norm estimates  lead to the boundedness of the Higgs field $\theta$ with respect to the Poincar\'e type metric.

\begin{thm}[\cite{Simpson1990}, Theorem 1 and \cite{Mochizuki2002}, Proposition 4.1]\label{thm_thetaestimate}
	Let $(E,\theta,h)$ be a nilpotent harmonic bundle on $X\backslash D$. Let $f_j,g_k$ be the matrix valued holomorphic functions as in (\ref{align_theta}). Then there exists a positive constant $C>0$ such that
	$$|f_j|_h\leq C(-\log |z_j|^2)^{-1},\quad\quad \textrm{for}\quad j=1,\dots,l;$$
	$$|g_k|_h\leq C,\quad\quad \textrm{for}\quad k=l+1,\dots,n.$$
	
	Therefore $$|\theta|_{h,\omega_P}\leq C$$
	holds on $U\backslash D$ for some admissible neighborhood $U$.
\end{thm}
{\color{red}

}

\section{$L^2$ cohomology and $L^2$ complex}
Let $(X,ds^2)$ be a complex hermitian manifold of dimension $n$ and $D$  a normal crossing divisor on $X$. Let $(E,h)$ be a hermitian holomorphic vector bundle on $X^\ast:=X\backslash D$. 
It gives us an inner product on the vector space of $E$-valued $(p, q)$-forms
$$\langle u,v\rangle_{h,ds^2}:=\int_{X^\ast}(u,v)_{h,ds^2}{\rm vol}_{ds^2}$$ where 
$(u,v)_{h,ds^2}$ is the pointwise inner product of $u$ and $v$ with respect to $h$ and $ds^2$. The $L^2$ norm of $u$ is defined as 
$$\|u\|_{h,ds^2}=\sqrt{\langle u,u\rangle_{h,ds^2} }.$$

Let $\sA^{p,q}_{X^\ast}$ denote the sheaf of smooth $(p,q)$-forms on $X^\ast$ for every $0\leq p,q\leq n$. Denote $\dbar:E\to E\otimes\sA^{0,1}_{X^\ast}$ to be the canonical $\dbar$ operator. Let $L^{p,q}_{(2)}(X^\ast,E;ds^2,h)$ (resp. $L^{m}_{(2)}(X^\ast,E;ds^2,h)$) be the space of square integrable $E$-valued $(p,q)$-forms (resp. $m$-forms) on $X^\ast$ with respect to the metrics $ds^2$ and $h$. Denote $\dbar_{\rm max}$ to be the maximal extension of the $\dbar$ operator defined on the domains
$$D^{p,q}_{X^\ast,ds^2}(E,h):=\textrm{Dom}^{p,q}(\dbar_{\rm max})=\{\phi\in L_{(2)}^{p,q}(X^\ast,E;ds^2,h)|\dbar\phi\in L_{(2)}^{p,q+1}(X^\ast,E;ds^2,h)\}.$$
Here $\dbar$ is taken in the sense of distribution.

The $L^2$-$\dbar$ cohomology $H_{(2),\rm max}^{p,\bullet}(X^\ast,E;ds^2,h)$ is defined as the cohomology of the complex
\begin{align}\label{align_L2_dol_cohomology}
D^{p,\bullet}_{X^\ast,ds^2}(E,h):=D^{p,0}_{X^\ast,ds^2}(E,h)\stackrel{\dbar_{\rm max}}{\to}\cdots\stackrel{\dbar_{\rm max}}{\to}D^{p,n}_{X^\ast,ds^2}(E,h).
\end{align}

Let $U\subset X$ be an open subset. Define $L_{X,ds^2}^{p,q}(E,h)(U)$ (resp. $L_{X,ds^2}^{m}(E,h)(U)$) to be the space of measurable $E$-valued $(p,q)$-forms (resp. $m$-forms) $\alpha$ on $U^\ast:=U\backslash D$ such that for every point $x\in U$, there is a neighborhood $V_x$ of $x$ so that
$$\int_{V_x\cap U^\ast}|\alpha|^2_{ds^2,h}{\rm vol}_{ds^2}<\infty.$$
For each $p$ and $q$, we define a sheaf $\sD_{X,ds^2}^{p,q}(E,h)$ on $X$ by
$$\sD_{X,ds^2}^{p,q}(E,h)(U):=\{\sigma\in L_{X,ds^2}^{p,q}(E,h)(U)|\bar{\partial}_{\rm max}\sigma\in L_{X,ds^2}^{p,q+1}(E,h)(U)\}$$
for every open subset $U\subset X$.
Define the $L^2$-Dolbeault complex of sheaves $\sD_{X,ds^2}^{p,\bullet}(E,h)$ as
\begin{align}\label{align_D_complex2}
\sD_{X,ds^2}^{p,0}(E,h)\stackrel{\dbar}{\to}\sD_{X,ds^2}^{p,1}(E,h)\stackrel{\dbar}{\to}\cdots\stackrel{\dbar}{\to}\sD_{X,ds^2}^{p,n}(E,h)
\end{align}
where $\dbar$ is defined in the sense of distribution.

\begin{lem}[fineness of the $L^2$ sheaf]\label{lem_fine_sheaf}
		Let $X$ be a compact complex manifold and $D$ a simple normal crossing divisor on $X$. Let $ds^2$ be a hermitian metric on $X^\ast:=X\backslash D$ which has Poincar\'e type growth along $D$ and let $(E,h)$ be a holomorphic vector bundle on $X^\ast$ with a possibly singular hermitian metric. Then $\sD^{p,q}_{X,ds^2}(E,h)$ is a fine sheaf for each $p$ and $q$.
	\end{lem}
	\begin{proof}
			Take an open subset $U\subset X$. It suffices to show that  $\eta\alpha\in\sD_{X,ds^2}^{p,q}(E,h)(U)$ for every $\alpha\in \sD_{X,ds^2}^{p,q}(E,h)(U)$ and every $\eta\in A^0_{\rm cpt}(U)$. Due to the asymptotic behavior of $ds^2$ (Proposition \ref{prop_poincare}), $|\eta|$ and $|\dbar\eta|_{ds^2}$ are  $L^\infty$ bounded on $U$. If we assume that $\alpha$ and $\dbar\alpha$ are locally $L^2$ integrable, then  $$\eta\alpha\in L_{X,ds^2}^{p,q}(E,h)(U),\quad \dbar(\eta\alpha)=\dbar\eta\wedge\alpha+\eta\wedge \dbar\alpha\in L_{X,ds^2}^{p,q+1}(E,h)(U).$$
			Thus the lemma is proved. 
	\end{proof}
\section{$L^2$ representation of the prolongation of Higgs bundles}
The purpose of this section is to establish a fine $L^2$ bi-complex resolution of the prolongation of Higgs bundles.
\subsection{Prolongation and parabolic structure}
Let $X$ be a complex  manifold of dimension $n$ and $D$  a normal crossing divisor on $X$. 
Let $(E,h)$ be a holomorphic vector bundle on $X\backslash D$ with a smooth hermitian metric $h$.  Let ${\bf a}=(a_1,\dots,a_l)\in \bR^l$
be a tuple of real numbers. For $\bm{b}=(b_1,\dots,b_l)\in\bR^l$, we denote $\bm{b}<\bm{a}$ if $b_i<a_i$ for every $i=1,\dots,l$.
\begin{defn}
	Let $X$ be a complex manifold and $D=\cup_{i=1}^l D_i$ a simple normal crossing divisor on $X$. A parabolic higgs bundle is a triple $(\{E_{\bm{a}}\},\theta)$ where
	for each $\bm{a}\in \bR^l$, $E_{\bm{a}}$ is a locally free coherent sheaf such that the following hold.
	\begin{itemize}
		\item  $ E_{\bm{a}+\bm{\epsilon}} = E{_{\bm{a}}}$ for any vector $\bm{\epsilon} = (\epsilon_1,\cdots,\epsilon_l)$ with $0<\epsilon_i\ll 1, \forall 1\leq i\leq l$.
		\item $E_{\bm{a}-{\bf 1}_i}={E_{\bm{a}}}\otimes \sO(-D_i)$ for every $1\leq i\leq l$. Here ${\bf 1}_i$ denotes $ (0,\dots,0,1,0,\dots,0)$ with 1 in
		the $i$-th component.
		\item The set of $\bm{a}$ such that ${\rm Gr}_{\bm{a}}E\neq 0$  is discrete in $\bR^l$.  Such $\bf{a}$ are called the weights.
		\item The Higgs field $\theta$ has at most logarithmic poles on $E_{\bm{a}}$, that is, $\theta$ can be extended to 
		\begin{align}
		E_{\bm{a}}\to E_{\bm{a}}\otimes\Omega_{X}(\log D)
		\end{align}
		for every $\bm{a}\in\bR^l$.
	\end{itemize} 
\end{defn}
\begin{defn}[Prolongation](Mochizuki\cite{Mochizuki2002}, Definition 4.2)\label{defn_prolongation}
	Let $U$ be an open subset of $X$ admissible to $D$. 	
	For any section $s\in \Gamma(U\backslash D,E)$,  let $|s|_h$ denote the norm function of $s$ with respect to the metric $h$. 
	We describe $|s|_h=O(\prod_{i=1}^l |z_i|^{-a_i})$ if there exists a
	positive number $C$ such that $$|s|_h\leq C\cdot \prod_{i=1}^l |z_i|^{-a_i}. $$
	
	We call $-\rm{ord}(s)\leq \bf{a}$ if $|s|_h=O(\prod_{i=1}^l |z_i|^{-a_i-\epsilon})$ for any positive number $\epsilon$.
	
    The $\sO_X$-module $ _{\bm{a}}E$ is defined as follows: For any open subset
	$U\subset X$,
	$$\Gamma(U, {_{\bm{a}}E}):=\{s\in\Gamma(U\backslash D,E)|-\textrm{ord}(s)\leq {\bm{a}}\}.$$
	The sheaf $_{\bf a}E $ is called the prolongment of $E$ by an increasing
	order $\bm{a}$.	
	Denote
	\begin{align}
		{\rm Gr}_{\bm{a}}E:={_{\bm{a}}}E/\cup_{\bm{b}<\bm{a}}{_{\bm{b}}}E.
	\end{align}	
\end{defn}
\begin{thm}[\cite{Mochizuki2009}, Proposition 2.53]\label{thm_parabolic}
	Let $X$ be a complex manifold and $D=\cup_{i=1}^l D_i$ a simple normal crossing divisor on $X$. Let $(E,\theta,h)$ be a tame harmonic bundle on $X\backslash D$.
	Then $(\{_{\bm{a}}E\},\theta)$ is a parabolic higgs bundle.
	The same conclusions hold for the flat bundle $(\cV,\nabla,h)$ associated with $(E,\theta,h)$ via Simpson's correspondence. 
\end{thm}
The following norm estimate for meromorphic sections is crucial in the proof of our main theorem.
\begin{thm}[\cite{Mochizuki20072}, Part 3, Chapter 13]\label{thm_tame_estimate}
	Let $(\cV,\nabla,h)$ be a tame harmonic bundle on $X^\ast=(\Delta^\ast)^l\times\Delta^{n-l}$. Let
	$$p:\bH^{l}\times \Delta^{n-l}\to (\Delta^\ast)^l\times \Delta^{n-l},$$
	$$(z'_1,\dots,z'_l,w_1,\dots,w_{n-l})\mapsto(e^{2\pi\sqrt{-1}z'_1},\dots,e^{2\pi\sqrt{-1}z'_l},w_1,\dots,w_{n-l})$$
	be the universal covering. Let ${\bf a}=(a_1,\dots,a_l)\in\bR^l$. Let
	$W^{(1)}=W(N_1),\dots,W^{(n)}=W(N_1+\cdots+N_n)$ be the residue weight filtrations on $V:=\Gamma(\bH^n,p^\ast\cV)^{p^\ast\nabla}$.
	Then
	for any $v\in V$ such that $$0\neq [v]\in {\rm Gr}_{l_n}^{W^{(n)}}\cdots{\rm Gr}_{l_1}^{W^{(1)}}V\cap {\rm Gr}_{\bf a}\cV,$$ one has
	\begin{align}\label{align_VHS_regular_sing}
		|v|_{h_\bV}\sim |s_1|^{-a_1}\cdots|s_l|^{-a_l} \left(\frac{\log|s_1|}{\log|s_2|}\right)^{l_1}\cdots\left(-\log|s_l|\right)^{l_n}
	\end{align}
	over any region of the form
	$$\left\{(s_1,\dots, s_l,w_1,\dots,w_{n-l})\in (\Delta^\ast)^l\times \Delta^{n-l}\bigg|\frac{\log|s_1|}{\log|s_2|}>\epsilon,\dots,-\log|s_l|>\epsilon,(w_1,\dots,w_{n-l})\in K\right\}$$
	for any $\epsilon>0$ and an arbitrary compact subset $K\subset \Delta^{n-l}$. The same conclusion holds for the holomorphic sections in the Higgs bundle $(E,\theta,h)$ associated with $(\cV,\nabla,h)$ via Simpson's correspondence.
\end{thm}
\subsection{$L^2$ representation}\label{section_L2_representation}
Let $X$ be a pre-compact open subset of a hermitian manifold $(M,\omega_M)$ and let $D=\cup_{i=1}^l D_i$ be a simple normal crossing divisor on $M$. Let $\sigma_i\in H^0(M,\sO_M(D_i))$ be the defining section of $D_i$ and $h_i$ an arbitrary hermitian metric on $\sO_M(D_i)$. We denote $\varphi_i=|\sigma_i|_{h_i}$. Let $\phi_P\in C^\infty(X\backslash D)$ such that $\omega_P:=\sqrt{-1}\ddbar\phi_P+\omega_M|_{X\backslash D}$ is a hermitian metric on $X\backslash D$ which has Poincar\'e type growth near $D\cap X$. Let $(E,\theta,h)$ be a nilpotent harmonic bundle on $M\backslash D$ together with $\dbar_E$ its holomorphic structure. For every indices ${\bf a}=(a_1,\dots,a_l)\in\bR^l$, denote $_{\bf a}E$ to be the associated prolongation of $(E,\theta)$.  We denote by $${\rm Dol}(_{\bf a}E,\theta):= {_{\bf a}}E\to {_{\bf a}}E\otimes\Omega_{X}(\log D)\to\cdots$$
the associated logarithmic Dolbeault complex.

For $N\in \bZ$, denote $h_N({\bf a}):=he^{-N\phi_P+\sum_{i=1}^l2a_i\log\varphi_i}$ to be a modified metric on $E$. Define the operator $D''=\dbar_E+\theta$. Denote by $\sD^{k}_{X,\omega_P}(E,D'',h_N({\bf a}))$ the sheaf of measurable $E$-valued $k$-forms $\alpha$ such that $\alpha$ and $D''\alpha$ are locally square integrable near every point of $X$.  Denote 
$$\sD^{\bullet}_{X,\omega_P}(E,D'',h_N({\bf a})):=\sD^{0}_{X,\omega_P}(E,D'',h_N({\bf a}))\stackrel{D''}{\to}\sD^{1}_{X,\omega_P}(E,D'',h_N({\bf a}))\stackrel{D''}{\to}\cdots$$
to be the associated $L^2$-Dolbeault complex. There is a natural inclusion $${\rm Dol}(_{\bf a}E,\theta)\to \sD^{\bullet}_{X,\omega_P}(E,D'',h_N({\bf a}+\bm{\epsilon}))$$ for every $\bm{0}<\bm{\epsilon}\in\bR^l$.

For every ${\bf a}=(a_1,\dots,a_l)\in\bR^l$, define
$$\sigma_{E,i}(a_i)=\min\{|b-a_i||b\in \sJ_{D_i}(E), b\neq a_i\},\quad \forall 1\leq i\leq l$$
and $\bm{\sigma}_E({\bf a})=(\sigma_{E,1}(a_1),\dots,\sigma_{E,l}(a_l))$.
The main result of this section is the following.
\begin{thm}\label{thm_main2}
	There is a constant $N_0$,  depending only on $(E,\theta,h)$ and $X$ (independent of $\bm{a}$ and $\bm{\epsilon}$ in the following), such that  the inclusion map $${\rm Dol}( _{{\bf a}}E,\theta)\to \sD^{\bullet}_{X,\omega_P}(E,D'',h_N({\bf a}+\bm{\epsilon}))$$ is a quasi-isomorphism for every
	${\bf a}\in\bR^l$, every $N>N_0$ and every $\bm{0}<\bm{\epsilon}<\bm{\sigma}_E({\bf a})$. 
\end{thm}
The proof of the theorem is postponed to the end of this section.
\subsection{Local behavior of the metrics}\label{section_local_behavior}
Since the statement of Theorem \ref{thm_main2} is local, we assume that $X_1=\Delta^n$ and $D=\sum_{i=1}^l D_i$ with $D_i=\{z_i=0\}$ for each $i$. Let $(E,h)$ be a nilpotent harmonic bundle over $X_1^\ast:=X_1\backslash D$.

For any $N\in \bZ_{> 0}$ and any ${\bf{a}}=(a_1,\cdots,a_l)\in \bR^l$, we define
\begin{align}
	\kappa({\bf{a}},N):=-N(\sum_{j=1}^l \log (-\log |z_j|^2)-\sum_{k=l+1}^n|z_k|^2)-\sum_{j=1}^l a_j {\rm{log}}|z_j|^2.
\end{align}
Set $h({\bf{a}},N):=h\cdot e^{-\kappa({\bf{a}},N)}$. Then
$$R(h({\bf{a}},N))=R(h)+\sqrt{-1}\partial\dbar \kappa({\bf{a}},N)=R(h)+N\omega_P.$$
\begin{rmk}\label{rmk_quasi}
	For the general settings as in \S 4.2, the modified hermitian metric  $h_{N}({\bf{a}}):=h\cdot (-\prod_{i=1}^l \log |\sigma_i|_{h_i}^2)^N\cdot \prod_{i=1}^l |\sigma_i|_{h_i}^{2a_i}$ is   locally quasi-isometric to $h({\bf{a}},N)$.
\end{rmk}
For every $1\leq i\leq n$, let $p_i$ be the projection from  $(\Delta^\ast)^l \times \Delta^{n-l}$ to its $i$-th factor. Note that $\Omega_{X_1^\ast}=\oplus_{i=1}^n L_i$  where $L_i$ is the trivial line bundle defined by $L_i:=p_i^\ast \Omega_{\Delta^\ast}$ for $i=1,...,l$ and $L_i=p_i^\ast\Omega_{\Delta}$ for $i=l+1,...,n$.   For any $p=0,...,n$, set $h_p$ to be the hermitian metric
on $T_{X_1^\ast}^p$ induced by $\omega_P$. Then there is a positive constant $C(p,l)>0$ depending
only on $p$ and $l$ so that $|R(h_p)|_{h_p,\omega_P}\leq C(p,l)$. Set $C_0:=\textrm{sup}_{p=0,...,n;l=1,...,n}C(p,l)$.
\begin{prop}\label{prop_nakanopositivity}
	Let $(E,\theta,h)$ be a nilpotent harmonic bundle over $X_1^\ast$. Then there exists a constant $N_0>0$ so that the following property holds after a possible shrinking of $X_1$.	
	For the  vector bundle $\mathcal{E}_p:=T_{X_1^\ast}^p\otimes E$  endowed with the metric $h_{\mathcal{E}_p}$ induced by $h({\bf{a}},N)$ and $\omega_P$, one has the following estimate
	\begin{align}\label{align_modifiedmetric}
		\sqrt{-1}R(h_{\mathcal{E}_p})\geq_{{\rm Nak}}  \omega_P\otimes {\rm Id}_{\mathcal{E}_p}
	\end{align}
	over $X_1^\ast$ for any $N\geq N_0$. Such $N_0$ does not depend on the choice of ${\bf{a}}$.
\end{prop}
\begin{proof}
	Since $(E,\theta,h)$ is a nilpotent harmonic bundle, $|\theta|_{h,\omega_P}$ is bounded after a possible shrinking of $X_1$ (Theorem \ref{thm_thetaestimate}). This, together with the formula  $R(h)=-[\theta,\theta_h^\ast]$ (\ref{align_selfdual_equ}), implies that  $|R(h)|_{h,\omega_P}\leq C$ for some positive constant $C$. Therefore $$\sqrt{-1}R(h)\geq_{\rm Nak}-C\omega_P\otimes {\rm Id}_{E}.$$
	
	We also know from $|R(h_p)|_{h_p,\omega_P}\leq C_0$ that $\sqrt{-1}R(h_p)\geq_{\rm Nak}-C_0\omega_P\otimes {\rm Id}_{T_{X_1^\ast}^p}$.
	
	Hence
	$$\sqrt{-1}R(h_p h)\geq_{\rm Nak}-(C+C_0)\omega_P\otimes {\rm Id}_{\sE_p}.$$
	
	Modifying the metric $hh_p$  to $h({\bf a},N)h_p$  on $\sE_p$, we obtain that 
	$$\sqrt{-1}R(h_ph({\bf a},N))\geq_{\textrm{Nak}}(N-C-C_0) \omega_P\otimes I_{\sE_p}.$$
	
	Taking $N\geq N_0=C+C_0+1$, we get the desired result.
\end{proof}
\subsection{Fine bi-complex resolution}
Since $|\theta|_{h_{N}({\bf a}),\omega_P}=|\theta|_{h,\omega_P}$ is bounded by Theorem \ref{thm_thetaestimate},  $$\theta:\sD_{X,\omega_P}^{p,q}(E,h_{N}({\bf{a}}))\rightarrow \sD_{X,\omega_P}^{p+1,q}(E,h_{N}({\bf{a}}))$$ is bounded for each $0\leq p,q\leq n$.

There is therefore the decomposition
\begin{align}\label{prop_decomp}
	\sD_{X,\omega_P}^{m}(E,D'',h_{N}({\bf{a}}))=\bigoplus_{p+q=m}\sD_{X,\omega_P}^{p,q}(E,h_{N}({\bf{a}})).
\end{align}
To construct the $L^2$ bi-complex resolution, we recall Demailly's formulation of  Hormander estimate to solving $\dbar$-equations on an incomplete K\"ahler manifold that admits a complete K\"ahler metric.
\begin{thm}[\cite{Demailly1982}, Theorem 4.1]\label{thm_demaillyestimate1}
	Let $X$ be a complete K\"ahler manifold with a (possibly) incomplete K\"ahler metric $\omega$. Let $(E,h)$ be a smooth hermitian vector bundle on $X$ such that
	$$\sqrt{-1}R(h)\geq_{{\rm Nak}} \epsilon \omega\otimes {\rm Id}_{\rm E},$$
	where $\epsilon>0$ is a positive constant. For every $q\geq 1$, assume that $g\in L_{(2)}^{n,q}(X,E;\omega,h)$ such that $\dbar g=0$. Then there exists $f\in L_{(2)}^{n,q-1}(X,E;\omega,h)$ so that $\dbar f=g$ and
	$$\|f\|_{h,\omega}^2\leq \frac{1}{\epsilon}\|g\|_{h,\omega}^2.$$	
\end{thm}
Now we are ready to prove the following proposition which gives the fine bi-complex resolution.
\begin{prop}\label{prop_localexactness1}
	Notations as in \S \ref{section_local_behavior}.
	Let $(E,h)$ be a hermitian holomorphic vector bundle on $X_1^\ast$.
	Denote by  $\omega_P=\sqrt{-1}\partial\dbar \phi_P$ the Poincar\'e metric on $X_1^\ast$ as in \S 2.2. Let $\tilde{h}:=h\cdot e^{-N\phi_P+\sum_{j=1}^l a_j {\log}|z_j|^2}$ be a modified metric on $E$ such that 
\begin{align}\label{align_positive}
	\sqrt{-1}R(\sE_p,h_p \otimes \tilde{h})\geq \omega_P\otimes {\rm Id}_{\sE_p},\quad\forall 0\leq p\leq n.
\end{align}
	Here $\sE_p=T_{X_1^\ast}^p\otimes E$ and  $h_{p}$ is the metric on $T^p_{X_1^\ast}$ induced by $\omega_P$.
	Then the complex 
	$$D_{X_1,\omega_P}^{p,0}(E,\tilde{h})\xrightarrow{\dbar} D_{X_1,\omega_P}^{p,1}(E,\tilde{h})\xrightarrow{\dbar}\cdots\xrightarrow{\dbar} D_{X_1,\omega_P}^{p,n}(E,\tilde{h})$$
	is exact at each $q\geq 1$ for every $p$.
\end{prop}
\begin{proof}
	First, notice that $X^\ast_1$ admits a complete K\"ahler metric by modifying $\omega_P$ to $\omega_P+\sqrt{-1}\partial\dbar\sum_{i=1}^n (1-|z_i|^2)^{-1}$. 
	Notice that 	
	\begin{align}\label{align_formtransfer}
		L_{(2)}^{n,q}(X_1^\ast,\mathcal{E}_{n-p};\omega_P,h_{n-p}\otimes \tilde{h})\simeq L_{(2)}^{p,q}(X_1^\ast,E;\omega_P,\tilde{h})
	\end{align}
	holds for any $0\leq p,q\leq n$.
	For any $q\geq 1$, any  $g\in 	L_{(2)}^{n,q}(X_1^\ast,\mathcal{E}_{n-p};\omega_P,h_{n-p}\otimes \tilde{h})$ with $\dbar g=0$, by  Theorem \ref{thm_demaillyestimate1} and (\ref{align_positive}) there is $f\in 	L_{(2)}^{n,q-1}(X_1^\ast,\mathcal{E}_{n-p};\omega_P,h_{n-p}\otimes \tilde{h})$ so that $\dbar f=g$.
	The proposition therefore follows from (\ref{align_formtransfer}).
\end{proof}
The following corollary is a direct consequence of the above proposition.
\begin{cor}\label{cor_exactness1}
Notations as in the beginning of \S \ref{section_L2_representation}. There is a constant $N_0$, depending only on $X$ and $(E,\theta,h)$, such that complex of sheaves
	$$\sD_{X,\omega_{P}}^{p,0}(E,h_{N}({\bf{a}}))\xrightarrow{\dbar} \sD_{X,\omega_{P}}^{p,1}(E,h_{N}({\bf{a}}))\xrightarrow{\dbar} \cdots\xrightarrow{\dbar} \sD_{X,\omega_{P}}^{p,n}(E,h_{N}({\bf{a}}))$$
	is exact at each $q\geq 1$ for every $p$ and every $N>N_0$.
\end{cor}
\begin{proof}
	If $x\notin D$, we can take an open neighborhood $U\subset X\backslash D$ of $x$ which is biholomorphic to a polydisk. Then the corollary follows from the usual $L^2$ Dolbeault lemma.
	It is therefore sufficient to consider an arbitrary point $x\in D$. Let $(U;z_1,\dots,z_n)$ be an admissible coordinate neighborhood which is biholomorphic to $\Delta_r^n$ for some $r\leq 1$ such that $U\backslash D\simeq(\Delta_r^\ast)^l\times \Delta_r^{n-l}$ for some positive integer $l$. 
	
	Note that $h_N({\bf a})$ is locally quasi-isometric to $\tilde{h}$ in Proposition \ref{prop_localexactness1}. 
	Choosing $N$ large enough that satisfies the conditions in Proposition \ref{prop_nakanopositivity}, the exactness is obtained from Proposition \ref{prop_localexactness1}. Notice that the compactness of $\overline{X}$ ensures the existence of the uniformed bound $N_0$.
\end{proof}
\subsection{Proof of Theorem \ref{thm_main2}}
Choosing $h_{N}({\bf{a}})$ that assures the validity of Corollary \ref{cor_exactness1}, we denote the sheaf
\begin{align*}
	(\Omega^m\otimes E)_{(2),N,{\bf a}}:={\rm Ker}\left(\sD_{X,\omega_P}^{m,0}(E,h_{N}({\bf{a}}))\xrightarrow{\dbar} \sD_{X,\omega_P}^{m,1}(E,h_{N}({\bf{a}}))\right),\quad \forall 0\leq m\leq n
\end{align*} 
and a complex of sheaves
$$Dol(E)^{\bullet}_{(2),N,{\bf a}}:=(\Omega^0\otimes E)_{(2),N,{\bf a}}\stackrel{{\theta}}{\to}(\Omega^1\otimes E)_{(2),N,{\bf a}}\stackrel{{\theta}}{\to}\cdots\stackrel{{\theta}}{\to}(\Omega^n\otimes E)_{(2),N,{\bf a}}.$$
That is, $(\Omega^m\otimes E)_{(2),N,{\bf a}}$ is the sheaf of germs of holomorphic sections $\sigma\in \Omega^m\otimes E$ which are locally square integrable with respect to  $h_{N}({\bf{a}})$ and $\omega_{P}$.

Then there exists a quasi-isomorphism
\begin{align}\label{align_bicomplex}
	\xymatrix{
		(\Omega^0\otimes E)_{(2),N,{\bf a}} \ar[r]^-{{\theta}}\ar[d]^{\dbar} & (\Omega^1\otimes E)_{(2),N,{\bf a}} \ar[r]^-{{\theta}}\ar[d]^{\dbar} &  (\Omega^2\otimes E)_{(2),N,{\bf a}} \ar[r]^-{{\theta}} \ar[d]^{\dbar} & \cdots\\
		\sD_{X,\omega_P}^{0,1}(E,h_{N}({\bf{a}}))\ar[r]^-{{\theta}}\ar[d]^{\dbar} & \sD_{X,\omega_P}^{1,1}(E,h_{N}({\bf{a}})) \ar[r]^-{{\theta}}\ar[d]^{\dbar} &
		\sD_{X,\omega_P}^{2,1}(E,h_{N}({\bf{a}}))
		\ar[r]^-{{\theta}}\ar[d]^{\dbar}&\cdots\\
		\sD_{X,\omega_P}^{0,2}(E,h_{N}({\bf{a}}))\ar[r]^-{{\theta}}\ar[d]^{\dbar} & \sD_{X,\omega_P}^{1,2}(E,h_{N}({\bf{a}})) \ar[r]^-{{\theta}}\ar[d]^{\dbar} &
		\sD_{X,\omega_P}^{2,2}(E,h_{N}({\bf{a}}))
		\ar[r]^-{{\theta}}\ar[d]^{\dbar}&\cdots\\
		\vdots&\vdots&\vdots\\
	}
\end{align}
By taking the total complex in (\ref{prop_decomp}), one gets
\begin{prop}\label{prop_quasi}
	The canonical morphism
	\begin{align}\label{align_quasiiso}
		Dol(E)^{\bullet}_{(2),N,{\bf a}}\to \sD^\bullet_{X,\omega_P}(E,D'',h_{N}({\bf{a}}))
	\end{align}
	is a quasi-isomorphism.
\end{prop}
Before calculating $Dol(E)^{\bullet}_{(2),N,{\bf a}}$, we would like to state a simple lemma that will be used.
\begin{lem}\label{lem_finite}
	When $N>1$,	$$\int_{\Delta^\ast_{\frac{1}{2}}}|z|^r({\rm{log}}|z|)^N{\rm vol}_{\omega_P}<+\infty$$ if and only if $r>0$.
\end{lem}
\begin{proof}
	Let $z=\rho e^{i\theta}$. Then 
	\begin{align}\label{align_integral}
	\int_{\Delta^\ast_{\frac{1}{2}}}|z|^r({\rm{log}}|z|)^N {\rm vol}_{\omega_P}=\pi\int_{0}^{\frac{1}{2}}\rho^{r-1}({\rm{log}}\rho)^{N-2}d\rho .
	\end{align}
	The right hand side of (\ref{align_integral}) is finite if and only if $r>0$.
\end{proof}
The proof of Theorem \ref{thm_main2} therefore follows from the following proposition.
\begin{prop}\label{prop_prolong}
  There is a constant $N_0$, depending only on $(E,\theta,h)$, such that
  $$Dol(E)^{\bullet}_{(2),N,{\bf a}+\bm{\epsilon}}= {\rm Dol}( {_{{\bf a}}}E,\theta)$$ for every ${\bf 0}<{\bm{\epsilon}}<\bm{\sigma}_E({\bf a})$ and $N>N_0$.
\end{prop}	
\begin{proof}
	Notations as in Theorem \ref{thm_tame_estimate}. For every $0\leq m\leq n$, let $\tilde{v}:= e_{i_1}\wedge\cdots \wedge e_{i_m}\otimes v$ be a holomorphic section of $\Omega^m\otimes E$ such that $$0\neq [v]\in {\rm Gr}_{l_n}^{W^{(n)}}\cdots{\rm Gr}_{l_1}^{W^{(1)}}E\cap {\rm Gr}_{\bf b}E$$ for some $l_1,\dots,l_n\in\bZ$ and some  ${\bf b}=(b_1,\dots,b_l)\in\bR^l$.    Here 
	\begin{align}
		e_{j}=\begin{cases}
			\frac{dz_j}{z_j}, & j=1,\dots,l \\
			dz_j, & j=l+1,\dots,m
		\end{cases}.
	\end{align}
 Then one has
	\begin{align}
		| \tilde{v}|_{h_N(\bm{a}+\bm{\epsilon}),\omega_P}\sim |s_1|^{a_1+\epsilon_1-b_1}\cdots|s_l|^{a_l+\epsilon_l-b_l} \big|\log|s_1|\big|^{\frac{N}{2}+l_1+\delta(1)}\big|\log|s_2|\big|^{\frac{N}{2}-l_1+l_2+\delta(2)}\cdots\big|\log|s_l|\big|^{\frac{N}{2}+l_n+\delta(n)}.
	\end{align}
Here \begin{align}\label{align_orthogonal_frame}
	\delta(i)=\begin{cases}
		1, & i\in\{i_1,\dots,i_m\}\cap \{1,\dots,l\} \\
		0, & {\textrm{otherwise}}
	\end{cases}.
\end{align}
	Taking $N_0=2\max\{|l_i|+2,|l_i-l_{i+1}|+2\}$. By Lemma \ref{lem_finite}, $$\int| \tilde{v}|_{h_N(\bm{a}+\bm{\epsilon}),\omega_P}^2{\rm vol}_{\omega_P}<\infty$$
	if and only if
	\begin{align}
		a_i+\epsilon_i-b_i>0,\quad i=1,\dots,l.
	\end{align}
	This is equivalent to that $v\in{_{{\bf a}}}E$ for any ${\bf 0}<{\bm{\epsilon}}<\bm{\sigma}_E({\bf a})$. This shows the proposition.
\end{proof}
\section{Application: Kawamata-Viehweg type vanishing theorem for Higgs bundles}
Let $X$ be an pre-compact open subset of an $n$-dimensional  projective K\"ahler manifold $(M,\omega)$ such that $X$ is a weakly pseudoconvex manifold and $\psi$ is a smooth exhausted plurisubharmonic function on $X$.  For every $c\in \bR$, denote $X_c:=\{x\in X|\psi(x)<c\}$. Let $D:=\cup_{i=1}^l D_i$ be a simple normal crossing divisor on $M$ and $(E,\theta,h)$ a nilpotent harmonic bundle on $M\backslash D$. Let $\sigma_i\in H^0(M,\sO_M(D_i))$ be the defining section of $D_i$. For each $1\leq i\leq l$, fix some smooth hermitian metric $h_i$ for the line bundle $\sO_M(D_i)$ so that $|\sigma_i|_{h_i}<1$. Since $\overline{X}$ is compact, the associated curvature tensor $R(h_i)$ is bounded for every $i$. We denote $\varphi_i=|\sigma_i|_{h_i}$.
Let $L$ be a holomorphic line bundle on $M$ so that some positive multiple $mL=A+F_1+F_2$ where $A$ is an ample line bundle, $F_1$ and $F_2$ are effective divisors supported in $D$ such that $A+F_1$ is a nef holomorphic line bundle. Suppose that $F_1=\sum_{i=1}^l \alpha_i D_i$ and $F_2=\sum_{i=1}^l r_i D_i$ where $\alpha_i, r_i\in \bZ_{\geq 0},\forall i$. Denote ${\bm{\alpha}}:=(\alpha_1,\dots,\alpha_l), {\bf{r}}:=(r_1,\dots,r_l)$.

The main technique that  we use to prove Theorem \ref{thm_main1} is Deng-Hao's
$L^2$ estimate for $D''$ operators on Higgs bundles following the pattern of Hormander's original $L^2$ estimate for $\dbar$ operators. 
\begin{prop}[\cite{DF2021}, Corollary 2.7]\label{prop_estimate}
	Let $(X,\omega)$ be a complete K\"ahler manifold and $(E,\theta,h)$  any harmonic bundle on $X$. Let $L$ be a line bundle on $X$ equipped with a hermitian metric $h_L$. Assume that for some $m>0$, one has
	\begin{align}
	\langle[\sqrt{-1}R(h_L),\Lambda_{\omega}]f,f\rangle_{\omega}\geq \epsilon |f|_{\omega}^2
	\end{align}
	for any $x\in X$, any $f\in \wedge^{p,q}T_{X,x}^\ast$ and any $p+q=m$. Set $(\widetilde{E},\widetilde{\theta},\widetilde{h}):=(E\otimes L,\theta\otimes {\rm{Id}}_L,h\otimes h_L)$ and $\widetilde{D''}=\dbar_{E\otimes L}+\widetilde{\theta}$. Then for any $v\in L_{(2)}^m(X,\widetilde{E};\omega,\widetilde{h})$ such that $\widetilde{D''} v=0$, there exists $u\in L_{(2)}^{m-1}(X,\widetilde{E};\omega,\widetilde{h})$ so that $\widetilde{D''}u=v$ and
	$$\|u\|^2\leq \frac{\|v\|^2}{\epsilon}.$$
\end{prop}
\subsection{Construction of the relavant metrics}
Since $A$ is ample, there exists a positive metric $h_A$ with weight $\varphi_A$ such that $\sqrt{-1}\partial\dbar \varphi_A\geq \epsilon_0 \omega$ for a certain positive constant $\epsilon_0$. Since $A+F_1$ is nef, for every positive constant $\epsilon$ there exists a metric $h_{\epsilon}$ on $A+F_1$ with weight $\varphi_{\epsilon}$ such that $\sqrt{-1}\partial\dbar \varphi_\epsilon\geq -\epsilon \omega$. Let $\varphi_{F_1}= \log |\sigma_i|_{h_i}^{2\alpha_i}$ (resp. $\varphi_{F_2}= \log |\sigma_i|_{h_i}^{2r_i}$) be the weight of a singular metric on $\sO(F_1)$ (resp. $\sO(F_2)$).

We define a singular metric $h_{\epsilon,\delta}$ on $L$ by the weight 
$$\varphi_L=\frac{1}{m}\left((1-\delta)\varphi_{\epsilon}+\delta(\varphi_A+\varphi_{F_1})+\varphi_{F_2}\right)$$
with $\epsilon\ll\delta\ll1$, $\delta$ rational. Then $\varphi_L$ has singularities along $F_1$ and $F_2$, and
\begin{align}\label{align_singularmetric}
	\sqrt{-1}\partial \dbar \varphi_L=&\frac{\sqrt{-1}}{m}\left((1-\delta)\partial \dbar \varphi_{\epsilon}+\delta \partial \dbar \varphi_A+\delta\partial\dbar \varphi_{F_1}+\partial\dbar \varphi_{F_2}\right)\\\nonumber
	\geq &\frac{1}{m}\left(-(1-\delta)\epsilon\omega+\delta( \epsilon_0\omega+[F_1])+[F_2]\right)\geq \frac{\delta}{m}\epsilon\omega
\end{align}
if we choose $\epsilon\leq\delta\epsilon_0$.
This is a singular metric which is smooth and positively curved outside $D$.

 Let $\psi_c:=\psi+\frac{1}{c-\psi}$ be a smooth exhausted plurisubharmonic function on $X_c$ and $\chi$ a convex increasing function. Choose a positive constant $N$. 
For any ${\bf{a}}=(a_1,\dots,a_l)\in \bR^l$, we modify the metric $h_{\epsilon,\delta}$ to
\begin{align}
h_{\epsilon,\delta}({\bf{a}}):=h_{\epsilon,\delta}\cdot \prod_{i=1}^l |\sigma_i|_{h_i}^{2a_i} \cdot (-\prod_{i=1}^l \log |\sigma_i|_{h_i}^2)^N
\end{align}
 and
\begin{align}
{h}_{\epsilon,\delta,c}({\bf{a}}):={h}_{\epsilon,\delta}({\bf{a}})\cdot e^{-\chi(\psi_c^2)}.
\end{align}
The associated curvatures are
\begin{align}\label{align_modifiedcurvatureterm1}
\sqrt{-1}R({h}_{\epsilon,\delta}({\bf{a}}))=&\sqrt{-1}R(h_{\epsilon,\delta})+N\sqrt{-1}\sum_{i=1}^l \frac{\partial \log |\sigma_i|_{h_i}^2\wedge \dbar \log |\sigma_i|_{h_i}^2}{(\log |\sigma_i|_{h_i}^2)^2}\\\nonumber
&+N\sqrt{-1}\sum_{i=1}^l \frac{R(h_i)}{\log |\sigma_i|_{h_i}^2}+\sqrt{-1}\sum_{i=1}^l a_i R(h_i)
\end{align}
and
\begin{align}\label{align_modifiedcurvatureterm}
\sqrt{-1}R({h}_{{\epsilon,\delta},c}({\bf{a}}))=\sqrt{-1}R({h}_{\epsilon,\delta}({\bf{a}}))+\sqrt{-1}\partial\dbar \chi(\psi_c^2).
\end{align}
Let $0< \gamma_1(x)\leq \cdot\cdot\cdot \leq \gamma_n(x)$ be the eigenvalues of $\sqrt{-1}R(h_{\epsilon,\delta})$ with respect to $\omega$ on $X\backslash D$. By the assumptions on $h_L$, $\gamma_i\geq \frac{\delta}{m}\epsilon$ holds on $X\backslash D$ for each $i$.
\begin{lem}\label{lem_modifiedmetric}
	Fixing $N$, we can pick ${\bf{a}}=(a_1,\dots,a_l)\in \bR^l$ such that {$|{\bf a}|:=\sum_{i=1}^l|a_i|$ is small enough} and rescale each $h_i$ by multiplying a small constant so that
	\begin{enumerate}
		\item \begin{align}\label{align_modifiedcurvaturepositivity1}
		\sqrt{-1}R({h}_{\epsilon,\delta}({\bf{a}}))\geq\sqrt{-1}R(h_{\epsilon,\delta}) -\frac{\delta\epsilon}{2nm}\omega\geq (1-\frac{1}{2n})\frac{\delta\epsilon}{m} \omega
		\end{align} on $X^\ast:=X\backslash D$.
		\item The metric  \begin{align}\label{align_metricconstruction}
		\omega_{N}({\bf{a}}):=\frac{\delta\epsilon}{2mn} \omega+\sqrt{-1}R({h}_{\epsilon,\delta}({\bf{a}}))
		\end{align}
		is a K\"ahler metric on $X^\ast$.
		\item \begin{align}\label{align_modifiedcurvaturepositivity}
		\sqrt{-1}R({h}_{{\epsilon,\delta},c}({\bf{a}}))\geq\sqrt{-1}R(h_{\epsilon,\delta})+\sqrt{-1}\partial\dbar \chi(\psi_c^2) -\frac{\delta\epsilon}{2mn}\omega\geq (1-\frac{1}{2n})\frac{\delta\epsilon}{m}\omega
		\end{align} on $X_c^\ast:=X_c\backslash D$.
		
		\item \begin{align}\label{align_metricconstruction2}
		\omega_{N,c}({\bf{a}}):=\omega_N ({\bf{a}})+\sqrt{-1}\partial\dbar \chi(\psi_c^2)
		\end{align}
		is a K\"ahler metric on $X_c^\ast$.
	\end{enumerate}
\end{lem}
\begin{proof}
	Once (1) is proved, (2) follows immediately. Taking the plurisubarmonicity of $\chi(\psi_c^2)$ into account, (3) and (4) can also be obtained. Therefore it suffices to prove (1).
	
	By  (\ref{align_modifiedcurvatureterm1}), the possible negative terms should appear in $N\sqrt{-1}\sum_{i=1}^l \frac{R(h_i)}{\log |\sigma_i|_{h_i}^2}+\sqrt{-1}\sum_{i=1}^l a_i R(h_i)$. Since $\overline{X}$ is compact and $N$ is fixed, we can  pick $a_i$ small enough and perturb $h_i$ to $\epsilon_0 \cdot h_i$ where $\epsilon_0$ is a small constant so that $$N\sqrt{-1}\sum_{i=1}^l \frac{R(h_i)}{\log |\sigma_i|_{h_i}^2}+\sqrt{-1}\sum_{i=1}^l a_i R(h_i)\geq -\frac{\delta\epsilon}{2mn}\omega.$$
	
	Thus the lemma is proved.
\end{proof}
\begin{rmk}
$\omega_{N}({\bf{a}})$ is a complete K\"ahler metric on $X^\ast$ and $\omega_{N,c}({\bf{a}})$ is a complete K\"ahler metric on $X_c^\ast$. They are both locally quasi-isometric the Poincar\'e type metric.
\end{rmk}
\begin{rmk}\label{rmk_quasiisometry}
	For a harmonic bundle $(E,\theta,h)$ on $X^\ast$, the modified hermitian metric  $h_{N,c}({\bf{a}}):=h_N({\bf{a}})\cdot e^{-\chi(\psi_c^2)}$ on $E$  is  locally quasi-isometric to $h({\bf{a}},N)$ in \S 4.3.
\end{rmk}
\begin{prop}\label{prop_positivecurvature}
	With the above notations, for every $c\in\bR$ one has
	\begin{align}
\langle[\sqrt{-1}R({h}_{{\epsilon,\delta},c}({\bf a})),\Lambda_{\omega_{N,c}({\bf a})}]f,f\rangle_{\omega_{N,c}({\bf a})}\geq \frac{1}{2}|f|_{\omega_{N,c}({\bf{a}})}^2
	\end{align}
	for any $x
	\in X_c^\ast$, any $p+q>n$ and any $f\in \wedge^{p,q}T_{X_c^\ast,x}^\ast$.
\end{prop}
\begin{proof}
	For any point $x\in X_c^\ast$, one can choose local coordinates $(z_1,\dots,z_n)$ around $x$ so that at $x$, $\omega=\sqrt{-1}\sum_{i=1}^n dz_i\wedge d\bar{z}_i$ and  $\sqrt{-1}R(h_{\epsilon,\delta,c}({\bf{a}}))=\sqrt{-1}\sum_{i=1}^n\tilde{\gamma}_i(x) dz_i\wedge d\bar{z}_i$  where $(1-\frac{1}{2n})\frac{\delta\epsilon}{m}\leq\tilde{\gamma}_1(x)\leq \tilde{\gamma}_2(x)\leq\cdots\leq \tilde{\gamma}_n(x)$ are the eigenvalues of $\sqrt{-1}R(h_{\epsilon,\delta,c}({\bf{a}}))$ with respect to $\omega$. Also  $\sqrt{-1}\partial\dbar \chi(\psi_c^2)=\sqrt{-1}\sum_{i=1}^n{\nu}_i(x) dz_i\wedge d\bar{z}_i$  with $0\leq {\nu}_1(x)\leq {\nu}_2(x)\leq \cdots\leq {\nu}_n(x)$ the associated eigenvalues with respect to $\omega$.
	
By (\ref{align_modifiedcurvaturepositivity}), $\tilde{\gamma}_i(x)\geq \gamma_i(x)+{\nu}_i(x)-\frac{\delta\epsilon}{2mn}$ for each $i$. Let $\lambda_1\leq \cdots \leq \lambda_n$ be the eigenvalues of $\sqrt{-1}R(h_{\epsilon,\delta,c}({\bf{a}}))$ with respect to $\omega_{N,c}({\bf{a}})$. Then $\lambda_i=\frac{\tilde{\gamma}_i}{\frac{\delta\epsilon}{2mn}+\tilde{\gamma}_i}$ by (3) and (4) in Lemma \ref{lem_modifiedmetric}. Thus $\frac{2n-1}{2n}
	\leq \lambda_i< 1$ for each $i=1,...,n$.
	
	We assume that $p\geq q$ without loss of generality. Then
	\begin{align}
	\langle[\sqrt{-1}R(h_{\epsilon,\delta,c}({\bf{a}})),\Lambda_{\omega_{N,c}({\bf{a}})}]f,f\rangle_{\omega_{N,c}({\bf{a}})}\geq& (\sum_{i=1}^p \lambda_i+\sum_{j=1}^q \lambda_j-\lambda_1-\cdot\cdot\cdot-\lambda_n)|f|_{\omega_{N,c({\bf{a}})}}^2\\\nonumber
	\geq&\left(p(\frac{2n-1}{2n})-(n-q)\right)|f|^2_{\omega_{N,c}({\bf{a}})}\\\nonumber
	\geq&\frac{1}{2}|f|^2_{\omega_{N,c}({\bf{a}})}.
	\end{align}
\end{proof}
\subsection{Fine resolution}
Denote $\widetilde{D''}:=\dbar_{E\otimes L}+\theta\otimes{\rm Id}_L$ to be the operator on $E\otimes L$.
Denote $\sD_{X,\omega_N({\bf{a}})}^\bullet (E\otimes L,\widetilde{D''},h_N({\bf{a}})\otimes h_{\epsilon,\delta})$ to be the complex 
	$$\sD^0_{X,\omega_{N}({\bf{a}})}(E\otimes L,\widetilde{D''},h_{N}({\bf{a}})\otimes h_{\epsilon,\delta})\xrightarrow{\widetilde{D''}}\sD^{1}_{X,\omega_{N}({\bf{a}})}(E\otimes L,\widetilde{D''},h_{N}({\bf{a}})\otimes h_{\epsilon,\delta})\xrightarrow{\widetilde{D''}}\cdots.$$
\begin{prop}\label{prop_quasiiso}
	Assume that $a_i-\frac{r_i}{m}\notin\sJ_{D_i}(E)$ for every $i=1,\dots,l$. Then the canonical morphism
	\begin{align}
		{\rm Dol}(_{-\frac{\bf r}{m}+{\bf a}}E,\theta)\otimes L\to \sD^\bullet_{X,\omega_{N}({\bf{a}})}(E\otimes L,\widetilde{D''},h_{N}({\bf{a}})\otimes h_{\epsilon,\delta})
	\end{align}
	is a quasi-isomorphism whenever $\delta>0$ is small enough.
\end{prop}
\begin{proof}
	Since $a_i-\frac{r_i}{m}\notin\sJ_{D_i}(E)$, there exists $\epsilon_i\in \bR_{>0}$ small enough such that $a_i-\frac{r_i}{m}-\epsilon_i\notin\sJ_{D_i}(E)$. Denote $\bm{\epsilon}=(\epsilon_1,\dots,\epsilon_l)\in\bR_{>0}^l$. Thus ${\rm Dol}(_{-\frac{\bf r}{m}+{\bf a}}E,\theta)={\rm Dol}(_{-\frac{\bf r}{m}+{\bf a}-\bm{\epsilon}}E,\theta)$.

	By Theorem \ref{thm_main2}, it suffices to show that
	\begin{align}\label{align_L2_L}
	\sD^\bullet_{X,\omega_{N}({\bf{a}})}(E\otimes L,\widetilde{D''},h_{N}({\bf{a}})\otimes h_{\epsilon,\delta})\simeq \sD^\bullet_{X,\omega_{N}({\bf{a}})}(E,D'',h_{N}(-\frac{{\bf r}+\delta\bm{\alpha}}{m}+{\bf{a}}))\otimes L	
	\end{align}
when $\delta$ is chosen small enough.
    By constructions one has 
    $$h_{N}({\bf{a}})\otimes h_{\epsilon,\delta}\sim h_{N}(-\frac{{\bf r}+\delta\bm{\alpha}}{m}+{\bf{a}})\otimes h_L$$ for some smooth hermitian metric $h_L$ on $L$. Let $\beta=\beta'\otimes e_L$ be a differential form with value in $E\otimes L$, where $\beta'$ is an $E$-valued form and $e_L$ is a holomorphic local generator of $L$. Since $|e_L|_{h_L}\sim 1$ and $|\theta|_{\omega_N(\bm{a})}$ is bounded, one obtains that $\|\beta\|<\infty, \|\widetilde{D''}(\beta)\|<\infty$ if and only if 
    $\|\beta'\|<\infty, \|D''(\beta')\|<\infty$. This proves (\ref{align_L2_L}).
    
\end{proof}
\subsection{Vanishing theorem on weakly pseudoconvex manifolds}
Now we are ready to prove the following vanishing theorem.
\begin{thm}\label{thm_main_KV}
	Let $X$ be a pre-compact open subset of an $n$-dimensional  projective manifold $M$ such that $X$ is a weakly pseudoconvex manifold  and $\psi$ is a smooth exhausted plurisubharmonic function on $X$. Let $D$ be a simple normal crossing divisor on $M$ and  $(E,\theta,h)$ a nilpotent harmonic bundle on $M\backslash D$. Suppose that some positive multiple $mL=A+F_1+F_2$ where $A$ is an ample line bundle, $F_1$ and $F_2$ are effective divisors supported in $D$ such that $A+F_1$ is a nef holomorphic line bundle. Assume that $F_2=\sum_{i=1}^l r_i D_i$ where $r_i\in \bZ_{\geq 0},\forall i$. Denote ${\bf{r}}:=(r_1,\dots,r_l)$ and
denote $X_c:=\{x\in X|\psi(x)<c\}$ for every $c\in \bR$.
Then
$$\bH^i(X_c,{\rm Dol}(_{-\frac{\bf r}{m}+{\bf a}}E,\theta)\otimes L)=0$$
holds for every $c\in \bR$, every $i>n$ and every  ${\bf a}=(a_1,\dots,a_l)\in \bR^l$ such that $|a_j|<\sigma_{E,j}(-\frac{r_j}{m})$ for every $j=1,\dots,l$.
\end{thm}
\begin{proof}
 	By Proposition \ref{prop_quasiiso}, there is a quasi-isomorphism 
 	$${\rm Dol}(_{-\frac{\bf r}{m}+{\bf a}}E,\theta)\otimes L\simeq_{\rm qis}\sD_{X,\omega_{N}({\bf{a}})}^\bullet (E\otimes L,\widetilde{D''},h_{N}({\bf{a}})\otimes h_{\epsilon,\delta})$$ under the assumption of the theorem.  Since $\sD_{X,\omega_N({\bf{a}})}^{p,q}(E\otimes L,\widetilde{D''},h_N({\bf{a}})\otimes h_{\epsilon,\delta})$ is a fine sheaf for each $p$ and $q$ (Lemma \ref{lem_fine_sheaf}), by (\ref{prop_decomp}) we obtain that $\sD_{X,\omega_N({\bf{a}})}^i(E\otimes L,\widetilde{D''},h_N({\bf{a}})\otimes h_{\epsilon,\delta})$ is a fine sheaf for each $i$. Notice that $\overline{X_c}$ is always compact since $\psi$ is exausted. For every $i>n$ and every $\alpha\in \sD_{X_c,\omega_{N}({\bf{a}})}^i (E\otimes L,\widetilde{D''},h_{N}({\bf{a}})\otimes h_{\epsilon,\delta})(X_c^\ast)$ such that $\widetilde{D''}\alpha=0$, it suffices to show that there exists $\beta\in \sD_{X_c,\omega_N({\bf{a}})}^{i-1}(E\otimes L,\widetilde{D''},h_N({\bf{a}})\otimes h_{\epsilon,\delta})(X_c^\ast)$ such that $\widetilde{D''}\beta=\alpha$.

 Note that $h_{N}({\bf{a}})\otimes h_{\epsilon,\delta}=h\otimes {h}_{\epsilon,\delta}({\bf{a}})$, $h_{N,c}({\bf{a}})\otimes h_{\epsilon,\delta}=h\otimes h_{\epsilon,\delta,c}({\bf{a}})$ and $\omega_{N}({\bf{a}})$ is incomplete on $X_c^\ast$. To apply Proposition \ref{prop_estimate}, we modify the metric $\omega_N({\bf{a}})$ to the complete K\"ahler metric $\omega_{N,c}({\bf{a}})$. Let  $\chi$  be a convex increasing function as in [\cite{Demailly2012}, Chapter VIII, Lemma (5.7)] such that
\begin{align}
	 \int_{X_c^\ast}|\alpha|_{\omega_{N,c}({\bf{a}}),h_{N,c}({\bf{a}})\otimes h_{\epsilon,\delta}}^2 {\rm vol}_{\omega_{N,c}({\bf a})}&\lesssim \int_{X_c^\ast}|\alpha|_{\omega_{N}({\bf{a}}),h_{N}({\bf{a}})\otimes h_{\epsilon,\delta}}^2e^{-\chi(\psi_c^2)} (1+\chi'\circ \psi_c^2+\chi''\circ \psi_c^2)^n{\rm vol}_{\omega_{N}({\bf a})}\\\nonumber &\lesssim\int_{X_c^\ast}|\alpha|_{\omega_{N}({\bf{a}}),h_N({\bf{a}})\otimes h_{\epsilon,\delta}}^2{\rm vol}_{\omega_N ({\bf{a}})}<+\infty.
\end{align}
Since $\omega_{N,c}({\bf{a}})$ is a complete K\"ahler metric on $X_c^\ast$, it follows from  Proposition \ref{prop_estimate}   and Proposition \ref{prop_positivecurvature} that there exists $\beta\in L_{(2)}^{i-1}(X_c^\ast,E\otimes L;\omega_{N,c}({\bf{a}}),h_{N,c}({\bf{a}})\otimes h_{\epsilon,\delta})(X_c^\ast)$ such that $\widetilde{D''}\beta=\alpha$. Then $\beta\in \sD_{X_c,\omega_N({\bf{a}})}^{i-1}(E\otimes L,\widetilde{D''},h_N({\bf{a}})\otimes h_{\epsilon,\delta})(X_c^\ast)$ since $\omega_N({\bf{a}})\sim\omega_{N,c}({\bf{a}})$ and $h_N({\bf{a}})\sim h_{N,c}({\bf{a})}$ locally on $X_c^\ast$. Thus the theorem is proved.	
\end{proof}
A relative version of the vanishing theorem follows from the above theorem.
\begin{thm}\label{thm_main_relative}
	Let $X$ be a  projective manifold of dimension $n$, $Y$ be an analytic space and $f:X\rightarrow Y$ be a proper surjective smooth map. Let $D=\sum_{i=1}^l D_i$ be a normal crossing divisor on $X$. Let $(E,\theta,h)$ be a nilpotent harmonic bundle on $X\backslash D$. Let $L$ be a holomorphic line bundle on $X$.
	Suppose that some positive multiple $mL=A+F_1+F_2$ where $A$ is an ample line bundle, $F_1$ and $F_2$ are effective divisors supported in $D$ such that $A+F_1$ is a nef holomorphic line bundle. Suppose that $F_2=\sum_{i=1}^l r_i D_i$ with $r_i\in \bZ_{\geq 0},\forall i$. Denote ${\bf{r}}:=(r_1,..,r_l)$.
	
	Then
	
	$$R^if_\ast({\rm Dol}(_{-\frac{\bf{r}}{m}+{\bf a}}E,\theta)\otimes L)=0$$
	for any $i>n$ and any ${\bf a}=(a_1,\dots,a_l)\in \bR^l$ such that $|a_j|<\sigma_{E,j}(-\frac{r_j}{m})$ for every $j=1,\dots,l$.
\end{thm}
\begin{proof}
	Since the problem is local, we may assume that $Y$ admits a non-negative smooth exhausted strictly plurisubharmonic function $\psi$ so that $\psi^{-1}\{0\}=\{y\}\subset Y$ is a point. To achieve this one may embed $Y$ into $\bC^N$ as a closed Stein analytic subspace and take $\psi=\sum_{i=1}^N\|z_i\|^2$. Since $f$ is proper, $f^\ast\psi$ is a smooth exhausted plurisubharmonic function on $X$.
	
    Fixing $c\in \bR_{>0}$, $f^{-1}Y_c$ is a weakly pseudoconvex manifold with $f^\ast \psi$ a smooth exhausted plurisubharmonic function on $f^{-1}Y_c$ and $\overline{f^{-1}Y_c}$ is compact. It follows from Theorem \ref{thm_main_KV}  that
	\begin{align}
	H^i(f^{-1}Y_d,\textrm{Dol}(_{-\frac{\bf{r}}{m}+{\bf a}}E,\theta)\otimes L)=0,\quad \forall i>n\quad \textrm{and}\quad\forall 0<d<c.
	\end{align}
 Taking $d\to 0$ we acquire that
	\begin{align}
	R^if_\ast(\textrm{Dol}(_{-\frac{\bf{r}}{m}+{\bf a}}E,\theta)\otimes L)_y=0,\quad \forall i>n.
	\end{align}
	This proves the theorem.
\end{proof}
\begin{rmk}
	By \cite{Mochizuki2009}, a parabolic stable higgs bundle admits a tame harmonic metric such that the parabolic structure coincide with the Simpson-Mochizuki's parabolic structure. Therefore Theorem \ref{thm_main_relative} implies Theorem \ref{thm_main1}.
\end{rmk}
\subsection{Generalization of Suh's vanishing theorem}
In this subsection we generalize Suh's Kawamata-Viehweg type vanishing theorem with coefficients for the Deligne extension of a variation of Hodge structure (\cite[Theorem 1]{Suh2018}). This generalization has two sides: one is that we treat polarized complex variations of Hodge structure rather than $\bQ$-polarized  or $\bR$-polarized variations of Hodge structure; the other is that we treat other prolongations so that we obtain the vanishing theorems for line bundles which may not be nef and big.
\begin{defn}
	Let $M$ be a complex manifold. A polarized complex variation of Hodge structure on $M$ of weight $k$ is a harmonic bundle $(\cV,\nabla,h)$ on $M$ which consists of a holomorphic flat bundle $(\cV,\nabla)$ and a harmonic metric, together with a decomposition $\cV\otimes_{\sO_M}\sA^0_{M}=\bigoplus_{p+q=k}\cV^{p,q}$ of $C^\infty$ bundles such that
	\begin{enumerate}
		\item The Griffiths transversality condition 
		\begin{align}\label{align_Griffiths transversality}
			\nabla(\cV^{p,q})\subset \sA^{0,1}_M(\cV^{p+1,q-1})\oplus \sA^{0,1}_M(\cV^{p,q})\oplus \sA^{1,0}_M(\cV^{p,q})\oplus \sA^{1,0}_M(V^{p-1,q+1})
		\end{align}
		holds for every $p$ and $q$. Here $\sA^{i,j}_M(\cV^{p,q})$ (resp. $\sA^{k}_M(\cV^{p,q})$) denotes the sheaf of $C^\infty$ $(i,j)$-forms (resp. $k$-forms) with values in $\cV^{p,q}$.
		\item The hermitian form $Q$ which equals $(-1)^{p}h$ on $\cV^{p,q}$ is parallel with respect to $\nabla$.
	\end{enumerate}
\end{defn}
We have a decomposition 
\begin{align}
	\nabla=\overline{\theta}+\dbar+\partial+\theta
\end{align}
according to  (\ref{align_Griffiths transversality}).

Let us fix some notations.
For every $p$, $F^p:=\oplus_{i\geq p}\cV^{i,k-i}\cap{\rm Ker}(\dbar+\overline{\theta})$ is a holomorphic subbundle which satisfies the Griffiths transversality condition that 
$\nabla(F^p)\subset F^{p-1}\otimes\Omega_M$ for every $p$. In literatures one usually regards $(\cV,\nabla,\{F^p\},h)$ as a polarized $\bC$-variation of Hodge structure. 

Due to (\ref{align_Griffiths transversality}), the graded quotient $({\rm Gr}_{F}\cV,{\rm Gr}_F \nabla)$ is a Higgs bundle in the sense of Simpson \cite[\S 8]{Simpson1988}. Denote $H:={\rm Ker}(\dbar:\cV\to\sA^1_M(\cV))$ and $H^{p,q}:=H\cap\cV^{p,q}$. Then $(H\simeq\oplus_{p+q=k}H^{p,q},\theta)$ is a Higgs bundle which is canonically isomorphic to $({\rm Gr}_{F}\cV,{\rm Gr}_F \nabla)$. Hence $({\rm Gr}_{F}\cV,{\rm Gr}_F \nabla)$ is related to $(\cV,\nabla)$ under Simpson's correspondence between flat connections and semistable Higgs bundles with vanishing Chern classes (\cite[\S 4]{Simpson1992}).
\begin{rmk}
	The Higgs bundle associated to a polarized complex variation of Hodge structure is always nilpotent.
\end{rmk}
Let $X$ be a complex manifold and $D=\cup_{i=1}^l D_i $ a simple normal crossing divisor on $X$. Let $(\cV,\nabla,\{F^p\}_{p\in\bZ})$ be a polarizable variation of Hodge structure of weight $k$.
For every ${\bm{a}}=(a_1,\dots,a_l)\in\bR^l$, let $\cV_{\geq\bm{a}}$ be the unique locally free $\sO_X$-module extending $\cV$ such that $\nabla$ induces a connection with logarithmic singularities 
$\nabla:\cV_{\geq{\bm{a}}}\to\cV_{\geq{\bm{a}}}\otimes\Omega_X(\log D)$ whose real parts of the eigenvalues of the residue of $\nabla$ along $D_i$ belong to $[{{a}}_i,{{a}}_i+1)$. Let $j:X\backslash D\to X$ be the open immersion.
Denote $F^p_{\geq{\bm{a}}}:=j_\ast F^p\cap\cV_{\geq{\bm{a}}}$.
The recent work of Deng \cite{Deng2022} shows the following theorem, which generalizes the nilpotent orbit theorem of Schmid \cite{Schmid1973} and Cattani-Kaplan-Schmid \cite{Cattani_Kaplan_Schmid1986}.
\begin{thm}\label{thm_degeneration_CVHS}
	For every ${\bm{a}}\in\bR^l$ and every $p\in\bZ$,
	$F_{\geq{\bm{a}}}^p$ is a subbundle of $\cV_{\geq{\bm{a}}}$ and $$F_{\geq{\bm{a}}}^p/F_{\geq{\bm{a}}}^{p+1}\simeq {_{{\bm{a}}}}H^{p,k-p},$$ 
	where ${_{{\bm{a}}}}H^{p,k-p}$ is Mochizuki's prolongation (a locally free $\sO_X$-modules) of the Hodge bundle $H^{p,k-p}:= F^p/F^{p+1}$.
	The connection $\nabla$ admits logarithmic singularities on $F_{\geq{\bm{a}}}^p$, i.e., $\nabla:F^p\to F^{p-1}\otimes\Omega_{X\backslash D}$ extends to 
	$$\nabla:F_{\geq{\bm{a}}}^p\to F_{\geq{\bm{a}}}^{p-1}\otimes\Omega_{X}(\log D).$$
\end{thm}
Thus one has a logarithmic de Rham complex
$$\cV_{\geq{\bm{a}}}\stackrel{\nabla}{\to} \cV_{\geq{\bm{a}}}\otimes\Omega_{X}(\log D)\stackrel{\nabla}{\to}\cV_{\geq{\bm{a}}}\otimes\Omega^2_{X}(\log D)\stackrel{\nabla}{\to}\cdots$$
and subcomplexes
$$F^{p}_{\geq{\bm{a}}}\stackrel{\nabla}{\to} F^{p-1}_{\geq{\bm{a}}}\otimes\Omega_{X}(\log D)\stackrel{\nabla}{\to}F^{p-2}_{\geq{\bm{a}}}\otimes\Omega^2_{X}(\log D)\stackrel{\nabla}{\to}\cdots,\quad\forall p.$$
Denote the graded quotient complex as $$_{\bf a}{\rm DR}_{(X,D)}(\cV,F):=\bigoplus_p F_{\geq{\bm{a}}}^p/F_{\geq{\bm{a}}}^{p+1}\stackrel{{\rm Gr}\nabla}{\to}\bigoplus_p F_{\geq{\bm{a}}}^{p-1}/F_{\geq{\bm{a}}}^{p}\otimes\Omega_{X}(\log D)\stackrel{{\rm Gr}\nabla}{\to}\cdots.$$
By Theorem \ref{thm_degeneration_CVHS} we have
\begin{align}
_{\bf a}{\rm DR}_{(X,D)}(\cV,F)\simeq{\rm Dol}({_{\bm{a}}}H,\theta)
\end{align}
where $(H,\theta,h)$ is the Higgs bundle associated to $(\cV,\nabla,h)$. Applying Theorem \ref{thm_main1} we obtain the following vanishing result.
\begin{thm}\label{thm_vanishing_VHS}
	Let $X$ be a  projective manifold of dimension $n$, $Y$ be an analytic space and $f:X\rightarrow Y$ be a proper surjective smooth map. Let $D=\sum_{i=1}^l D_i$ be a normal crossing divisor on $X$. Let $(\cV,\nabla,\{\cV^{p,q}\}_{p+q=k},h)$ be a polarizable variation of Hodge structure of weight $k$ on $X\backslash D$. Let $L$ be a holomorphic line bundle on $X$.
	Suppose that some positive multiple $mL=A+F_1+F_2$ where $A$ is an ample line bundle, $F_1$ and $F_2$ are effective divisors supported in $D$ such that $A+F_1$ is a nef holomorphic line bundle. Suppose that $F_2=\sum_{i=1}^l r_i D_i$ where $r_i\in \bZ_{\geq 0},\forall i$. Denote ${\bf{r}}:=(r_1,\dots,r_l)$.
	
	Then
	$$R^if_\ast\left(_{-\frac{\bf{r}}{m}+{\bf a}}{\rm DR}_{(X,D)}(\cV,F)\otimes L\right)=0$$
	for any $i>n$ and any ${\bf a}=(a_1,\dots,a_l)\in \bR^l$ such that $|a_j|<\sigma_{E,j}(-\frac{r_j}{m})$ for every $j=1,\dots,l$.
\end{thm}
To consider the canonical extension of $\cV$, let us take ${\bm{a}}=(0,\dots,0)$. In this case $_{\bf{0}}{\rm DR}_{(X,D)}(\cV,F)$ is the logarithmic de Rham complex associated with the canonical extension of $(\cV,F)$ considered in \cite{Suh2018}. Taking $Y={\rm pt}$, $F_2=0$ and ${\bm{a}}=(0,\dots,0)$, Theorem \ref{thm_vanishing_VHS} implies Suh's vanishing theorem (\cite[Theorem 1]{Suh2018}).

\bibliographystyle{plain}
\bibliography{suh}

\end{document}